\documentclass[12pt]{amsart}
\usepackage{a4wide}

\usepackage{amsthm}
\usepackage{amsmath}
\usepackage{amsxtra}
\usepackage{stmaryrd}
\usepackage{dsfont}
\usepackage{enumitem}
\usepackage{units}
\usepackage{amssymb}
\usepackage{graphicx}
\usepackage{caption}
\usepackage{subcaption}

\usepackage{float}

\usepackage{tikz}
\usetikzlibrary{positioning,arrows,calc}
\tikzset{
modal/.style={>=stealth’,shorten >=1pt,shorten <=1pt,auto,node distance=1.5cm,
semithick},
world/.style={circle,draw,minimum size=0.5cm,fill=gray!15},
point/.style={circle,draw,inner sep=0.5mm,fill=black},
reflexive above/.style={->,loop,looseness=7,in=120,out=60},
reflexive below/.style={->,loop,looseness=7,in=240,out=300},
reflexive left/.style={->,loop,looseness=7,in=150,out=210},
reflexive right/.style={->,loop,looseness=7,in=30,out=330}
}

\newtheorem{teo}{Theorem}[section]
\newtheorem{lemma}[teo]{Lemma}

\newtheorem{defi}[teo]{Definition}
\newtheorem{coro}[teo]{Corollary}
\newtheorem{fatto}[teo]{Fact}

\newtheorem{observation}[teo]{Observation}

\newcommand{\ZFC}{\ensuremath{\mathsf{ZFC}}}
\newcommand{\ZF}{\ensuremath{\mathsf{ZF}}}
\newcommand{\Z}{\overline{\mathsf{ZF}}}
\newcommand{\IZF}{\ensuremath{\mathsf{IZF}}}
\newcommand{\nff}{\mathrm{NFF}}
\newcommand{\bq}{\mathrm{BQ}_\varphi}
\newcommand{\dom}{\mathrm{dom}}
\newcommand{\ran}{\mathrm{ran}}

\newcommand{\VA}[1]{\mathbf{V}^{(#1)}}
\newcommand{\A}{\mathbb{A}}
\newcommand{\T}{\mathbb{T}}

\newcommand{\one}{\mathbf{1}}
\newcommand{\zero}{\mathbf{0}}
\newcommand{\V}{\mathbf{V}}
\newcommand{\co}{\mathsf{Cobounded}}
\newcommand{\s}{\mathbb{PS}_3}
\newcommand{\ps}{\mathrm{PS}_3}
\newcommand{\w}{\mathcal{W}^+}

\newcommand{\llaw}{\mathsf{LL}_\varphi}
\newcommand{\ba}{\mathrm{BA}}
\newcommand{\pa}{\mathrm{PA}}
\newcommand{\ha}{\mathrm{HA}}
\newcommand{\lb}{\llbracket}
\newcommand{\rb}{\rrbracket}

\title{$\ZF$ and its interpretations}
\author{S. Jockwich Martinez, S. Tarafder, G. Venturi}
\email{santijoxi@hotmail.com\\ souravt09@gmail.com \\gio.venturi@gmail.com}
\address{University of Campinas (UNICAMP), Bar\~ao Geraldo, SP, 13083-896, Brazil.}
\address{St. Xavier's College, 30 Mother Teresa Sarani, Kolkata 700016, India.}

\begin{document}

\maketitle

\begin{abstract}
 In this paper, we unify the study of classical and non-classical algebra-valued models of set theory, by studying variations of the interpretation functions for $=$ and $\in$. Although, these variations coincide with the standard interpretation in Boolean-valued constructions, nonetheless they extend the scope of validity of $\ZF$ to new algebra-valued models. This paper presents, for the first time, non-trivial paraconsistent models of full $\ZF$. Moreover, due to the validity of Leibniz's law in these structures, we will show how to construct proper models of set theory by quotienting these algebra-valued models with respect to equality,  modulo the filter of the designated truth-values.   
   
\end{abstract}

\tableofcontents

\section*{Introduction}

Forcing was introduced in 1963 by Paul Cohen. The revolution that this method brought to set theory cannot be understated. The many different models (or universes) of set theory that have been created since the Sixties had the effect of putting into question the central role of set theory as the foundations of mathematics, suggesting a relativity of the truth-value of many set-theoretical sentences. 

A reformulation of the forcing method in terms of Boolean-valued constructions was soon after proposed, during the Sixties, by Scott and, independently, Vop\v{e}nka, building on the algebraic work of Rasiowa and Sikorski. The two methods yield equivalent independence results, because of the bridge that the Stone duality theorem offers between partially ordered sets and Boolean algebras.  

The models of $\ZF$ that we obtain by forcing, as it is the case for Boolean-valued models, thus represent many different ways in which the universe of sets could be. The different and heterogeneous structures we obtain by this method constitute an embarrassment of riches that is responsible for a heated debate over which represent(s) the correct view on set theory.\footnote{For a general overview of the many positions in this debate see \cite{Antos2015-ANTMCI}, \cite{Venturiforcing}, and
\cite{Maddy2016-MADSF}.}  

In this paper, we show that this debate should not be limited only to the non-logical component of the first-order theory $\ZF$, since the class of structures that validate the axioms of $\ZF$ greatly exceeds the classical boundaries.\footnote{See \cite{TarVen2020} for the effect that the construction of non-classical models of (fragments of) $\ZF$ has on the notion of independence.} 
Indeed we will prove, in Theorem \ref{PAZF} and Lemma  \ref{lemma: (T, BA) is paraconsistent}, that there are non-trivial structures that satisfy simultaneously $\ZF$, a formula $\varphi$, and its negation $\neg \varphi$.
Moreover, in contrast to the paraconsistent models of set theory that we find in the literature, the ones constructed in this paper show that $\ZF$ holds consistently; meaning that their negation do not hold. 

The originality of this result stems from the novelty of the proposed approach. Indeed, the key ingredient of Theorem \ref{PAZF} is a syntactic modification of the axioms of $\ZF$ and of the interpretation of the basic relations of the language of set theory: equality ($=$) and set-membership ($\in$). These modifications  in a whole depend on the following propositional equivalence: $(A \to B) \equiv (A\to B) \land (\neg B \to \neg A)$.  

The reason for these changes is twofold. On the one hand, such a simple modification allows to overcome the limits of the previous approaches to algebra-valued models.
Indeed, it is only with the new interpretation function developed in this paper that we can validate all instances of the Separation, Collection, and Foundation axiom schema as given in Figure \ref{figure: Aximos of ZF}, together with  Leibniz's Law of indiscernibility of identicals. Moreover, it is this modified interpretation function that will force us to change the Extensionality axiom by supplementing it with its contrapositive (classically equivalent) counterpart. On the other hand, to resort to contrapositive versions of classical elements of Boolean-valued models has a deeper conceptual motivation. As a matter of fact, in this paper we find ourselves operating in a 
weak logical context, where classical equivalences may cease to hold. Therefore, information about sets which are classically encoded in the axioms and the interpretation of the basic relations of set theory may get lost in this new logical environment. Since a logic that tolerates contradictions should necessarily give a separate account of assertion and denial, it is then the task of the axiomatizer to describe both set-membership and non-set-membership. If this is done, in the classical case, only by committing to asserting what belongs to what (and then using the duality of logic to extend it to denial), this is not anymore the case within this weaker logical environment.

Besides the technical and conceptual motivations for the justification of the constructions presented in this paper, it is also worth noting that from a purely classical perspective these changes are innocuous. As a matter of fact, the resulting set theory is only a classically equivalent version of $\ZF$, whose equivalence can be already proved within classical propositional logic.

Another interesting point is that in changing the interpretation of equality and set-membership we do not change the underlying mathematical ontology. As a matter of fact, the definition of an algebra-valued universe $\VA{\A}$, based on an algebra $\A$, does not depend on the assignment function that we use to give a truth-value to the formulas of the extended language of set theory (as in the standard case, a language equipped with constants for every element of  $\VA{\A}$). This means that the change in the assignment function corresponds only to a change of interpretation of set theory and not to which sets there are. But then, the assignment functions described in this paper suggest that we can have different ways to interpret $\ZF$, that coincide in the classical case but diverge in non-classical contexts. 



Disregarding the intrinsic value one can attach to non-classical logics, this paper not only presents new ways to build models of $\ZF$, but it also shows that Boolean-valued models represent only a (small) fragment of a much larger class of models of set theory (classical and not).


The paper is organized as follows: In \S 1, we review some preliminaries and we introduce the main objects of study. In \S 2, we review the results on the validity of $\ZF$ under the standard assignment function as defined, for example, in \cite{Bell}. The core of the paper, \S 3, is devoted to the proof of Theorem \ref{PAZF} and to the study of a new assignment function, called $\lb \cdot \rb_\pa$. 
Then, in \S 4, we study the internal logics of the algebra-valued models constructed using the new and the standard assignment function. In \S 5, we show how to build  quotient models out of the algebra-valued models constructed using the new assignment function. In \S 6, we compare the results of this paper with other models of paraconsistent set theory. We end the paper with \S 7, where we present some open questions which might suggest new and exciting research paths.

\section{Algebra-valued models}

\subsection{Algebraic structures}

We introduce a few algebraic preliminaries to fix the notations. 


\begin{defi}
An algebraic structure is a triplet $\langle \mathbf{A}, \mathbf{O}, \mathbf{C} \rangle$, where $\mathbf{A}$ is a non-empty set, $\mathbf{O}$ is a collection of operations on $\mathbf{A}$, and $\mathbf{C}$ is a collection (possibly empty) of some elements of $\mathbf{A}$, called constants.
\end{defi}

\begin{defi}
An algebraic structure $\langle \mathbf{A}, \mathbf{O}, \mathbf{C} \rangle$ is called a lattice if 
$\mathbf{O}$ contains two binary operators $\wedge$ and $\vee$ satisfying the following laws for any $a, b, c \in \mathbf{A}$:
\begin{enumerate}
    \item[(i)] commutative: $a \wedge b = b \wedge a$, ~ $a \vee b = b \vee a$,
    \item[(ii)] associative: $a \wedge (b \wedge c) = (a \wedge b) \wedge c$, ~ $a \vee (b \vee c) = (a \vee b) \vee c$,
    \item[(iii)] idempotent: $a \wedge a = a$, ~ $a \vee a = a$, and
    \item[(iv)] absorption: $a \wedge (a \vee b) = a$, ~ $a \vee (a \wedge b) = a$;
\end{enumerate}

\smallskip

\noindent $\langle \mathbf{A}, \mathbf{O}, \mathbf{C} \rangle$ is called a distributive lattice if in addition both the following properties hold:\\
\\
\hspace*{0.4 cm} (v) $a \wedge (b \vee c) = (a \wedge b) \vee (a \wedge c)$, ~ $a \vee (b \wedge c) = (a \vee b) \wedge (a \vee c)$.

\end{defi}

\noindent A \emph{partial order relation} $\leq$ is defined in a lattice as follows:
\[a \leq b \mbox{ if and only if } a \wedge b = a \mbox{ (or equivalently } a \vee b = b), \mbox{ for any } a, b \in \mathbf{A}.\]
An element $a \in \mathbf{A}$ is called an upper bound of a subset $S$ of $\mathbf{A}$, if $x \leq a$, for all $x \in S$. Similarly, an element $b \in \mathbf{A}$ is called a lower bound of a subset $S$ of $\mathbf{A}$, if $x \geq b$, for all $x \in S$.
\begin{defi}
A lattice $\langle \mathbf{A}, \mathbf{O}, \mathbf{C} \rangle$ is said to be bounded if $\mathbf{C}$ contains two elements, that we denote by,  $\one$ and $\zero$ such that $\one$ is the upper bound of $\mathbf{A}$ and $\zero$ is the lower bound of $\mathbf{A}$. The elements $\one$ and $\zero$ are called the top and bottom elements of the lattice, respectively.
\end{defi}

\begin{defi}
A bounded lattice $\langle \mathbf{A}, \mathbf{O}, \mathbf{C} \rangle$ is said to be complete if for any $S \subseteq \mathbf{A}$,
\begin{enumerate}
    \item[(i)] there is an element $a \in \mathbf{A}$ such that $a$ is an upper bound of $S$ and for any upper bound $p$ of $S$, $a \leq p$, in this case $a$ is called the supremum of $S$, denoted by $\bigvee S$, and
    \item[(ii)] there is an element $b \in \mathbf{A}$ such that $b$ is a lower bound of $S$ and for any lower bound $q$ of $S$, $q \leq b$, in this case $b$ is called the infimum of $S$, denoted by $\bigwedge S$.
\end{enumerate}
\end{defi}

\noindent From now on, we shall denote the structure of a bounded lattice as $\langle \mathbf{A}, \wedge, \vee, \one, \zero \rangle$.

\subsection{Assignment functions and validity of a sentence}\label{subsection: assignment function}

In this section, we shall recall the general construction of algebra-valued models of set theory and define what an assignment function is. Let us begin with an algebraic structure $\langle \mathbf{A}, \mathbf{O}, \mathbf{C} \rangle$, which contains a complete distributive lattice $\langle \mathbf{A}, \wedge, \vee, \one, \zero \rangle$ as a substructure, and a language $\mathcal{L}_{\in}$ for set theory having the same structure of the algebra, together with the two binary predicate symbols $\in$ and $=$ and the two quantifiers $\exists$ and $\forall$.

Let us consider $\mathbf{V}$ as the class of all sets. Then, by transfinite recursion, we construct the following:
\begin{align*}
\mathbf{V}_{\alpha}^{(\A)} &= \{x\,:\mbox{ $x$ is a function and }\mathrm{ran}(x) \subseteq \mathbf{A}\\
&\hspace{15mm}\mbox{ and there is $\xi < \alpha$ with }\mathrm{dom}(x)
\subseteq \mathbf{V}_{\xi}^{(\A)}\}\mbox{ and}\\
\mathbf{V}^{(\A)} &= \{x\,:\,\exists \alpha (x \in \mathbf{V}_{\alpha}^{(\A)})\}.
\end{align*}

Let $\mathcal{L}_\A$ be the extended language of $\mathcal{L}_{\in}$, constructed by adding constant symbols corresponding to every element in $\VA{\A}$. To increase the readability, the name corresponding to each $u \in \VA{\A}$ will be denoted by the symbol $u$ itself, in the extended language $\mathcal{L}_\A$.

\begin{defi}
A function $\lb \cdot \rb^\A$ from the collection of all sentences of $\mathcal{L}_\A$ into $\mathbf{A}$ is said to be an assignment function.
\end{defi}

Given an assignment function, the validity of the sentences in $\mathcal{L}_\A$ depends on a subset of  $\mathbf{A}$ called the \emph{designated set}. Its elements represent the possible truth-values of valid sentences.

\begin{defi}
Let $\A = \langle \mathbf{A}, \wedge, \vee, \one, \zero \rangle$ be a bounded lattice. A set $D \subseteq \mathbf{A}$ is said to be a designated set if
\begin{enumerate}
\item[(i)] $\one \in D$, but $\zero \notin D$,
\item[(ii)] if $x \in D$ and $x \leq y$, then $y \in D$, and
\item[(iii)] for any $x, y \in D$, $x \wedge y \in D$.
\end{enumerate}
\end{defi}
\noindent Thus, a designated set is nothing but a \emph{filter} in $\A$.

\begin{defi}
A sentence $\varphi$ of $\mathcal{L}_\A$ is said to be valid in $\VA{\A}$ with respect to a given assignment function $\lb \cdot \rb^\A$ and a designated set $D$,  whenever $\lb \varphi \rb^\A \in D$. We denote this fact by $\VA{\A, ~ \lb \cdot \rb^\A} \models_D \varphi$ (or simply $\VA{\A} \models \varphi$, when the algebra $\A$, the assignment function $\lb \cdot \rb^\A$, and the designated set $D$ are clear from the context).
\end{defi}
\noindent From now on, we use the notation $\lb \cdot\rb$ to refer to the assignment function $\lb \cdot\rb^\A$, when the algebra $\A$ is clear from the context.

Notice that, the validity of a set theoretic sentence not only depends on the designated set, but also on the assignment function. 
Indeed, different assignment functions may have different interpretations for $\in$ and $=$. Then, these differences can be extended recursively to all formulas of set theory.

\subsection{Assignment functions with two different interpretations for $\in$ and $=$}

Let $\A = \langle \mathbf{A}, \wedge, \vee, \Rightarrow, ^*, \one, \zero \rangle$ be an algebra, where $\langle \mathbf{A}, \wedge, \vee, \one, \zero \rangle$ is a complete bounded distributive lattice. Let $\mathcal{L}_{\in}$ be the standard language of set theory, having the same structure as $\A$. We shall denote by $\wedge, \vee, \to, \neg, \top$, and $\bot$ the basic logical connectives. 
We then recursively define two different assignment functions: $\lb \cdot \rb_\ba$  and $\lb \cdot \rb_\pa$ which interpret $\in$ and $=$  as follows: 
\begin{enumerate}
    \item[(i)] \begin{align*}
\lb u \in v \rb_\ba&= \bigvee_{x \in \mathrm{dom}(v)} \big{(}v(x) \land \lb x = u \rb_\ba \big{)},\\
\lb u = v \rb_\ba &= \bigwedge_{x \in \mathrm{dom}(u)} \big{(}u(x) \Rightarrow \lb x \in v \rb_\ba\big{)} \land \bigwedge_{y \in \mathrm{dom}(v)} \big{(}v(y) \Rightarrow \lb y \in u \rb_\ba \big{)};
\end{align*}
    \item[(ii)] \begin{align*}
\lb u \in v \rb_\pa&= \bigvee_{x \in \mathrm{dom}(v)} \big{(}v(x) \land \lb x = u \rb_\pa \big{)},\\
\lb u = v \rb_\pa &= \bigwedge_{x \in \dom(u)} \big{(}(u(x) \Rightarrow  \lb x \in v \rb_\pa) \wedge (\lb  x \in v\rb_\pa^* \Rightarrow  u(x)^*)\big{)}\\
& \qquad \qquad \wedge \bigwedge_{y \in \dom(v)} \big{(}(v(y) \Rightarrow  \lb y \in u \rb_\pa) \wedge ( \lb y \in  u \rb_\pa^* \Rightarrow  v(y)^*  )\big{)};
\end{align*}  
\end{enumerate}

Then, these maps homomorphically extend to other formulas in the obvious way:  
for any two closed well-formed formulas $\varphi$ and $\psi$ in $\mathcal{L}_\A$,
{\allowdisplaybreaks
\begin{align*}
\lb \top \rb_X & = \one,\\
\lb \bot \rb_X & = \zero,\\
\lb \varphi \land \psi \rb_X & = \lb \varphi \rb_X \land \lb \psi \rb_X,\\
\lb \varphi \lor \psi \rb_X & = \lb \varphi \rb_X \lor \lb \psi \rb_X,\\
\lb \varphi \rightarrow \psi \rb_X & = \lb \varphi \rb_X \Rightarrow \lb \psi \rb_X,\\
\lb \neg \varphi \rb_X & = \lb \varphi \rb_X^*,\\
\lb \forall x \varphi(x) \rb_X &= \bigwedge_{u \in \mathbf{V}^{(\A)}} \lb \varphi(u) \rb_X \mbox{, and}\\
\lb \exists x \varphi(x) \rb_X &= \bigvee_{u \in \mathbf{V}^{(\A)}} \lb \varphi(u) \rb_X,
\end{align*}}where throughout this extension $\lb \cdot \rb_X$ is a given fixed assignment function, among $\lb \cdot \rb_\ba$ and $\lb \cdot \rb_\pa$. The suffixes `$\ba$', `$\pa$' and `$\ha$'  of the notations $\lb \cdot \rb_\ba$, $\lb \cdot \rb_\pa$ and $\lb \cdot \rb_\ha$ stand for `Boolean assignment function', `paraconsistent assignment function', and  `Heyting assignment function', respectively. The justifications of these notations are portrayed in the following sections.

\section{The assignment function $\lb \cdot \rb_\ba$ and $\ZF$}\label{subsection: Boolean-valued models}

The Zermelo-Fraenkel $(\ZF)$ axiom system, in the language $\mathcal{L}_\in$ is displayed below in Figure \ref{figure: Aximos of ZF}. In the schemes, $\varphi$ is a formula with $n+2$ free variables. This formulation follows closely \cite{Bell}. This definition of $\ZF$ is classically equivalent to the standard one that we find in logic textbooks and it is chosen in order to simplify the task of checking the validity of the axioms in algebra-valued models.

\begin{figure}[h!]
\begin{align*}
\tag{\textsf{Extensionality}}
&\forall x \forall y \big{(}\forall z (z \in x \leftrightarrow z \in y) \rightarrow x = y\big{)}\\
\tag{\textsf{Pairing}}
&\forall x \forall y \exists z \forall w \big{(}w \in z \leftrightarrow (w = x \lor w = y)\big{)}\\
\tag{\textsf{Infinity}}
&\exists x \big{(}\exists y (\forall z ~ \lnot(z\in y) \land y\in x )\land \forall w (w \in x \to \exists u (u \in x \wedge w \in u))\big{)}\\
\tag{\textsf{Union}}
& \forall x \exists y \forall z \big{(}z \in y \leftrightarrow \exists w (w \in x \wedge z \in w)\big{)}\\
\tag{\textsf{Power Set}}
&\forall x \exists y \forall z \big{(}z \in y \leftrightarrow \forall w (w \in z  \to w \in x)\big{)}\\
\tag{\textsf{Separation}$_\varphi$}
& \forall x \exists y \forall z \big{(}z \in y \leftrightarrow (z \in x \land \varphi(z))\big{)}\\
\tag{\textsf{Collection}$_\varphi$}
&\forall x \big{(}\forall y(y \in x \to \exists z \varphi (y, z))\\
&\hspace{2.5cm} \rightarrow \exists w \forall v(v \in x \to \exists u(u \in w \wedge \varphi (v, u)))\big{)}\\
\tag{\textsf{Foundation}$_\varphi$}
&\forall x \big{(}(\forall y(y \in x \to \varphi(y)) \rightarrow \varphi(x)) \rightarrow \forall z \varphi(z)\big{)}
\end{align*}
\caption{The axioms of $\ZF$.}\label{figure: Aximos of ZF}
\end{figure}

\subsection{Boolean and Heyting-valued models}

We start by recalling the standard results on the validity of set theory in algebra-valued models constructed using Boolean or Heyting algebras. 

\begin{teo}[cf.\ \cite{Bell}, Theorem 1.33]
Let $\A$ be any complete Boolean algebra and $\{\one\}$ be the designated set, where $\one$ is the top element of $\A$, then  $\VA{\A, ~ \lb \cdot \rb_\ba} \models \ZF$.\footnote{In this paper, we will not deal with the Axiom of Choice, leaving this issue for a sequel of this work.}
\end{teo}

Following a similar construction, Grayson proved in 1977 that the assignment function $\lb \cdot \rb_\ba$ also works to validate $\IZF$, intuitionistic Zermelo Fraenkel set theory, in Heyting-valued models.


\begin{teo}[\cite{grayson}, p.~410]
Let $\A$ be any complete Heyting algebra and $\{\one\}$ be the designated set, where $\one$ is the top element of $\A$, then  $\VA{\A, ~ \lb \cdot \rb_\ba} \models \IZF$.
\end{teo}

\subsection{Generalized algebra-valued models}\label{subsection: generalized algebra-valued models}

The study of algebra-valued models was generalized to a bigger class of algebras in \cite{Loewe}. 
In that paper, the authors defined the negation-free fragment ($\nff$) to be the closure of the atomic formulas by the connectives $\wedge$, $\vee$, $\to$, $\bot$, and the quantifiers $\forall$, $\exists$. 
The negation-free fragment of $\ZF$ is denoted by $\nff$-$\ZF$, and similarly, $\nff$-$\IZF$ stands for the negation-free fragment of $\IZF$.
Note that, the two formulas $\lnot \varphi$ and $\varphi \to \bot$ are classically and  intuitionistically equivalent, where $\varphi \to \bot$ is a negation-free formula but $\lnot \varphi$ is not. Hence, for any sentence $\psi$ of $\mathcal{L}_\in$, $\ZF \models \psi$ iff $\nff$-$\ZF \models \psi'$ and $\IZF \models \psi$ iff $\nff$-$\IZF \models \psi'$, where $\psi'$ is the formula replacing every subformula of the form $\neg \varphi$ of $\psi$ by $\varphi \to \bot$.
The fragment of $\ZF$ and $\nff$-$\ZF$ without the $\mathsf{Foundation~Axiom~Schema}$ will be denoted by $\ZF^-$ and $\nff$-$\ZF^-$, respectively. 

The (meta-)induction principle for an algebra-valued model $\VA{\A}$ states that, for any property $\Phi$ of names, if for any $x \in \VA{\A}$, we have the property: 
\begin{center}
    if for all $y \in \dom(x)$, $\Phi(y)$ hold, then $\Phi(x)$ also holds,
\end{center}
then for all $x \in \VA{\A}$, $\Phi(x)$ holds. This induction principle is used through out this paper to develop the algebra-valued models.

\begin{defi}[\cite{Loewe}, p.~194]
An algebra
$\A = \langle \mathbf{A}, \wedge, \vee, \Rightarrow, \mathbf{1}, \mathbf{0} \rangle$, where the substructure $\langle \mathbf{A}, \wedge, \vee, \mathbf{1}, \mathbf{0} \rangle$  is a complete distributive lattice, is said to be a \emph{deductive reasonable implication algebra}
if for any $x, y, z \in \mathbf{A}$, the following properties hold:
\begin{description}
\item[P1] $(x\wedge y)\leq z$ implies $x\leq (y\Rightarrow z)$,

\item[P2] $y\leq z$ implies $(x\Rightarrow y)\leq (x\Rightarrow z)$, 

\item[P3] $y\leq z$ implies $(z\Rightarrow x)\leq (y\Rightarrow x)$, and

\item[P4] $((x\wedge y)\Rightarrow z) = (x\Rightarrow (y\Rightarrow z))$.
\end{description}
\end{defi}

For any formula $\varphi(x)$ with one free variable $x$ in $\mathcal{L}_{\in}$ and any $u \in \VA{\A}$, consider the following equation:
\begin{align*}
\lb\forall x\big{(}x\in u \to \varphi(x)\big{)}\rb_\ba =
\bigwedge_{v\in\mathrm{dom}(u)} \big{(}u(v)\Rightarrow
\lb\varphi(v)\rb_\ba\big{)}. \tag{$\bq^\ba$}
\end{align*}
Then, the following theorem holds.

\begin{teo}[\cite{Loewe}, Theorem 3.3 and 3.4]\label{thm: RIA valued model of NFF-ZF}
If $\mathbf{V}$ is a model of $\ZF$ and $\A$ is a deductive reasonable implication algebra
so that $\bq^\ba$
holds in $\VA{\A}$ for every negation-free formula $\varphi$, then for any designated set $D$, $\VA{\A, ~ \lb \cdot \rb_\ba} \models_D \nff$-$\ZF^-$.
\end{teo}

If $\A$ is any complete Boolean algebra or complete Heyting algebra, then $\A$ is not only proved to be a deductive reasonable implication algebra, but $\bq^\ba$ holds in $\VA{\A}$ for every negation-free formula $\varphi$. But are there non-Boolean and non-Heyting algebras which satisfy the conditions of Theorem \ref{thm: RIA valued model of NFF-ZF}? Interestingly, this question has a positive answer.

\subsection{$\mathcal{T}$-valued models}

In \cite{VenSan}, the authors defined a class of complete totally ordered algebras $\mathcal{T}$ (i.e. algebras whose underlying set is totally ordered), able to satisfy the hypotheses of Theorem \ref{thm: RIA valued model of NFF-ZF} in the case of the choice of an \emph{ultrafilter} for $\T \in \mathcal{T}$: $D_\T = \{a \in \T: \zero < a\}$, where $\Rightarrow$ and $^*$ are defined as follows: for any two elements $a$ and $b$ of $\T$,
\begin{equation*}
a \Rightarrow b = \left \{ \begin{array}{ll}
                        \zero, & \mbox{if } a \neq \zero \mbox{ and } b = \zero; \\
                        \one, & \mbox{otherwise,}
                      \end{array}\right.
\end{equation*}
\[\one^* = \zero,~ \zero^* = \one, \mbox{ and } a^* = a, \mbox{ if } a \neq \one, \zero. \]
Notice that, for every $\T \in \mathcal{T}$ the ultrafilter $D_\T$ is fixed.
Therefore, for any 
$\mathbb{T} \in \mathcal{T}$, we have that $\VA{\mathbb{T}, ~ \lb \cdot \rb_\ba} \models_{D_\mathbb{T}} \nff$-$\ZF^-$.\footnote{See \cite{VenSanJPL} for a discussion on the different negations, and their properties, that can be defined in $\mathcal{T}$-valued models.}


\begin{defi} Let  $\mathcal{W}$ be the subclass of $\mathcal{T}$, which consists of all well-ordered algebras: i.e., algebras whose underlying set is order-isomorphic to an ordinal. On the other hand, $\w \subseteq \mathcal{W}$ is the collection of all algebras which are not order-isomorphic to the successor of any limit ordinal. An element of $\mathcal{W}$ will be called a $\mathcal{W}$-algebra, whereas an element of $\w$ will be called a $\w$-algebra.
\end{defi}

It is also proved in \cite{VenSan} that, for every negation-free formula $\varphi$ and $\T \in \mathcal{W}$, $\VA{\mathbb{T}, ~ \lb \cdot\rb_\ba} \models_{D_\mathbb{T}} \mathsf{Foundation}_\varphi$. Hence, the following theorem holds.

\begin{teo}[\cite{VenSan}, Corollary 4.10]\label{theorem: T valued models under BA} For every $\mathbb{T} \in \mathcal{W}$, $\VA{\mathbb{T}, ~ \lb \cdot \rb_\ba} \models_{D_\T} \nff$-$\ZF$.
\end{teo}

 In addition, we have the following theorem.

\begin{teo}[\cite{Taraf}, Theorem 22]
For any $\mathbb{T} \in \mathcal{W}$ and any designated set $D$ of $\T$, $\VA{\mathbb{T}, ~ \lb \cdot \rb_\ba} \models_{D} \mathsf{Axiom~of~Choice}$, if $\mathbf{V} \models \mathsf{Axiom~of~Choice}$.
\end{teo}

\subsection{Designated $\co$-algebra}

In this section, we introduce a class of structures, the (designated) \textsf{Cobounded} algebras,  that generalizes the class of $\mathcal{W}^+$-algebras. 

\begin{defi}
An algebra $\A = \langle \mathbf{A}, \wedge, \vee, \Rightarrow, \one, \zero \rangle$ is called a $\co$-algebra if
\begin{enumerate}
    \item[(i)] $\langle \mathbf{A}, \wedge, \vee,\one, \zero \rangle$ is a complete distributive lattice,
    \item[(ii)] if $\bigvee\limits_{i \in I}a_i = \one$, then there exists $j \in I$ such that $a_j = \one$, where $I$ is an index set and $a_i \in A$ for all $i \in I$,
    \item[(iii)] if $\bigwedge\limits_{i \in I}a_i = \zero$, then there exists $j \in I$ such that $a_j = \zero$, where $I$ is an index set and $a_i \in A$ for all $i \in I$, 
    \item[(iv)] the binary operator  $\Rightarrow$ is defined as follows: for $a, b \in A$
\begin{align*}
    a \Rightarrow b & = \left \{ \begin{array}{ll}
                        \zero, & \mbox{if } a \neq \zero \mbox{ and } b = \zero; \\
                        \one, & \mbox{otherwise.}
                      \end{array}\right.
\end{align*}
\end{enumerate}
\end{defi}

Notice that a $\co$-algebra does not contain the unary operator $^*$, which is considered as the algebraic interpretation of the unary connective $\lnot$ of $\mathcal{L}_\in$. We then extend the structure of $\co$-algebra by introducing a unary operator $^*$ with the help of a designated set.

\begin{defi}\label{Definition: MTV algebra}
Let $\mathbb{A} = \langle \mathbf{A}, \wedge, \vee, \Rightarrow,^*, \one, \zero \rangle$ be an algebra and $D$ be a designated set of $\A$. The pair $(\mathbb{A},D)$ is called a
designated $\mathsf{Cobounded}$-algebra  if 
\begin{enumerate}
    \item[(i)] $\langle \mathbf{A}, \wedge, \vee, \Rightarrow, \one, \zero \rangle$ is a $\mathsf{Cobounded}$-algebra,
    \item[(ii)] the operator $^*$ is defined as follows: for $a, b \in A$
\begin{align*}
   a^{*} & = \left \{\begin{array}{ll}
                        \zero, & \mbox{if } a = \one;  \\
                        a, & \mbox{if } a \in D \setminus \{\one\};\\
                        \one, & \mbox{if } a \notin D.
                        \end{array}\right.
\end{align*}
\end{enumerate}
A designated $\mathsf{Cobounded}$-algebra $(\A, D)$ is called an ultra-designated $\mathsf{Cobounded}$-algebra, if the corresponding designated set $D$ is an ultrafilter.
\end{defi}

The main result of this paper, Theorem \ref{PAZF}, will be proved with respect to  $\mathsf{Cobounded}$-algebra-valued models.

\begin{fatto}
Every $\mathsf{Cobounded}$-algebra is a deductive reasonable implication algebra, but the converse is not true in general.
\end{fatto}

\begin{proof}
The proof follows from the definitions. On the other hand, a counter example for the converse part could be any complete Boolean algebra, having more than two elements.
\end{proof}

\begin{fatto} 
Every $\mathsf{Cobounded}$-algebra $\mathbb{A}$ has a unique atom (i.e. an element $a \in \A$ for which there is no $x \in \A$ such that $\zero < x < a$) and a unique coatom (i.e. an element $c \in \A$ for which there is no $x \in \A$ such that $c < x < \one$). 
\end{fatto}

\begin{proof} Let  $S=\mathbf{A} \setminus \{\one\}$, where $\mathbf{A}$ is the underlying set of $\A$. Then by the definition of $\co$-algebra, $\bigvee S \neq \one$. Let $c = \bigvee S$. Then by the construction, $c$ is a coatom. It is unique as well, otherwise the join of two different coatoms would be $\one$, violating the definition of $\co$-algebra.

Similarly, we can prove that if $X = \mathbf{A} \setminus \{\zero\}$, then $\bigwedge X$ is the unique atom of $\A$.
\end{proof}

\begin{lemma}\label{lemma: ultrafilter of an MTV algebra}
Let $(\A, D)$ be an ultra-designated $\mathsf{Cobounded}$-algebra. Then, $\mathbf{A} \setminus D = \{\zero\}$, where $\mathbf{A}$ is the underlying set of $\A$.
\end{lemma}

\begin{proof}
Suppose otherwise, so there exists a non-zero element $a \in \mathbf{A} \setminus D$. Consider an element $d \in D$ and suppose $b = (d \wedge a)$. 
Since $a \neq \zero \neq d$, by the definition of designated $\mathsf{Cobounded}$-algebra $b \neq \zero$. Now, $b \leq a$ implies that $b \notin D$, otherwise by the definition of filter $a$ would also be a member of $D$. Since $b \leq d$ and $b \notin D$, the filter $D$ can be extended to another filter which includes the element $b$, which contradicts that $D$ is an ultrafilter.
\end{proof}

\subsection{Examples of designated $\mathsf{Cobounded}$-algebras}

We now show that the collection of designated $\mathsf{Cobounded}$-algebras is a proper class and that the corresponding logics can vary from classical to non classical. 

{\bf Example 1}. The smallest designated $\mathsf{Cobounded}$-algebra, is a two elements algebra consisting of just the top and the bottom elements of a lattice. In this case, the designated set consists of the top element only. When restricted to two elements the operations of a designated $\mathsf{Cobounded}$-algebra corresponds to those of the two-element Boolean algebra. Hence, its corresponding logic is classical propositional logic.

{\bf Example 2}. Any complete distributive lattice having a unique atom and a unique co-atom can be a designated $\mathsf{Cobounded}$-algebra with respect to any designated set and with the proper selection of the operators $\Rightarrow$ and $^*$.
To have a concrete example, consider any complete Boolean algebra (or Heyting algebra) $\mathbb{B} =\langle \mathbf{B}, \wedge, \vee, \Rightarrow_\mathbf{B}, ^{*_\mathbf{B}}, \top, \bot \rangle$, where $\top$ and $\bot$ are respectively the top and bottom elements of $\mathbb{B}$.
Then, add tow elements $\one$ and $\zero$ to the set $\mathbf{B}$ and extend the two operators $\wedge$ and $\vee$ on the resultant set $\mathbf{A}$ (say) as
\[\one \wedge a = a, ~\zero \wedge a = \zero, ~\one \vee a = \one, \mbox{ and } \zero \vee a = a,\]
for all $a \in \mathbf{B}$. Then, $\langle \mathbf{A}, \wedge, \vee, \one, \zero\rangle$ becomes a complete distributive lattice, so that $\one$ and $\zero$ are respectively the top and bottom elements, whereas $\top$ and $\bot$ becomes the unique co-atom and atom, respectively in the extended structure.
Let us now consider any designated set $D$ of $\langle \mathbf{A}, \wedge, \vee, \one, \zero\rangle$ and define two new operators $\Rightarrow$ and $^*$ on $\mathbf{A}$ as in Definition \ref{Definition: MTV algebra}(v). Then, $(\A, D)$ is a designated $\mathsf{Cobounded}$-algebra, where $\A = \langle \mathbf{A}, \wedge, \vee, \Rightarrow, ^*, \one, \zero\rangle$. Hence, the underlying set of a Boolean or Hetying algebra when stretched with a further top and bottom element and equipped with an implication and a negation as in Definition \ref{Definition: MTV algebra} becomes a designated $\mathsf{Cobounded}$-algebra.


{\bf Example 3}. Any $\w$ algebra $\T$ is a designated $\co$-algebra if the designated set is taken as the ultrafilter $D_\T$. Moreover, for any $\w$ algebra $\T$ and any designated set $D$ of it, if we modify the unary operator $^*$ of $\T$ acording to Definition \ref{Definition: MTV algebra} then $(\T, D)$ becomes a designated $\co$-algebra.

\subsection{Main goal of the paper}

One can notice that in the algebra-valued models studied in \cite{Loewe}, \cite{VenSan}, \cite{illoyal}, \cite{TarVen2020}, \cite{Taraf}, \cite{VenTar} etc.
Leibniz's law of indiscernability of identicals fails under the $\lb \cdot \rb_\ba$-assignment function. Concretely, this means that there exist formulas $\varphi(x)$ (which contain the negation symbol) and elements $u$ and $v$ of such algebra-valued models for which 
$\lb u = v \rb_\ba$ and $\lb \varphi(u) \rb_\ba$ are valid, but for which $\lb \varphi(v) \rb_\ba$ is invalid.
Moreover, there exist instances of $\mathsf{Separation}_\varphi$, $\mathsf{Collection}_\varphi$, and $\mathsf{Foundation}_\varphi$ (involving formulas $\varphi$'s which contain negation) which are invalid in $\VA{\mathbb{T}}$ under the assignment function $\lb \cdot \rb_\ba$.

\medskip

\emph{Leibniz's law of indiscernibility of identicals}: 
\begin{equation}
    \forall x \forall y \big{(} (x = y \wedge \varphi(x)) \to \varphi(y)\big{)}
    \tag{$\llaw$}
\end{equation}
We say that $\llaw$ holds in a model, if $\llaw$ is valid there for any formula $\varphi(x)$ of the language. 
It was proved in Theorem 2.12 of  \cite{VenTarSAN3} that $\mathsf{Separation}$ and $\llaw$ fail in every $\mathbb{T}$-valued model in which the domain of $\mathbb{T}$ has more than two elements and considering $D$ a set of designated values with more than one element.

Although the relation 
\[u \sim v \mbox{ if and only if } \lb u = v \rb_\ba \mbox{ is valid},\]
is an equivalence relation between the elements of any $\mathcal{W}$-algebra-valued model $\VA{\mathbb{T}}$, however, the failure of $\llaw$ imposes a dramatic limitation to the construction of a proper ontology for the non-classical set theory of $\VA{\mathbb{T}}$. As a matter of fact, the failure of $\llaw$ prevents one from being able to build a quotient model $\VA{\mathbb{T}}/\!\sim$ in which the interpretation of the equality relations lines up with the equality between the elements of the model. This is an important step in the construction of Boolean-valued models that, unfortunately fails in the construction of $\mathcal{W}$-algebra-valued models.


The main goal of the paper is to solve these problems.
Concretely, we will devise  a new assignment function which can validate, in any \textsf{Cobounded}-algebra-valued model: 
\begin{enumerate}
    \item[(i)] $\llaw$ for all formulas $\varphi$, and
    \item[(ii)] the axioms of a  classically equivalent version of $\ZF$, that we will call $\overline{\ZF}$,
\end{enumerate}
 Moreover, we will show that the resulting algebra-valued models can even be lifted to well-defined quotient models of the corresponding set theory.

This is what motivated the definition of the  $\lb \cdot \rb_\pa$-assignment function. Indeed, Theorem \ref{theorem: validity of LL} will show the validity of  $\llaw$ in \textsf{Cobounded}-algebra-valued models, under the $\lb \cdot \rb_\pa$-assignment function, while Theorem \ref{PAZF} the validity of $\overline{\ZF}$. 

The only difference between $\ZF$ and $\overline{\ZF}$ consists in the reformulation  of $\mathsf{Extentionality}$ in a classically equivalent form. The reason for this reformulation is not only conceptual, but also mathematical.

\begin{observation}\label{observation: failure of Extensionality in PA}
For any designated $\mathsf{Cobounded}$-algebra $(\mathbb{A},D)$, where the domain of $\A$ contains more than two elements, $\VA{\A,~\lb \cdot \rb_\pa} \not\models_{D_\mathbb{A}} \mathsf{Extensionality}$.
\end{observation}

\begin{proof}
Arbitrarily fix two elements $u = \{( w, a )\}$ and $v = \{(w, \one)\}$, where $w \in \VA{\A}$ and $a$ be an element of $\A$ so that $\zero < a < \one$. Then, $\lb \forall z (z \in u \leftrightarrow z \in v) \rb_\pa \in D_{\A}$ but it can be checked that $\lb u = v \rb_\pa = \zero \notin D_{\A}$.
\end{proof}

On closer inspection, the failure of $\mathsf{Extensionality}$ in $\VA{\A,~\lb \cdot \rb_\pa}$ is not surprising. Indeed, the interpretation of $=$ under  $\lb \cdot \rb_\pa$ introduces new conjuncts which are the contrapositive counterparts of those considered in $\lb \cdot \rb_\ba$.

For this reason, we are forced to reformulate $\mathsf{Extensionality}$ to match this new interpretation function. 

\begin{equation}
    \forall x \forall y \forall z \big{(}((z \in x \leftrightarrow z \in y) \wedge (\neg(z \in x) \leftrightarrow \neg(z \in y))) \to x = y \big{)}.
    \tag{$\mathsf{\overline{Extensionality}}$}
\end{equation}
\smallskip

\begin{defi}
Let $\overline{\ZF}$ stand for the axiom system $\ZF - \mathsf{Extensionality} + \mathsf{\overline{Extensionality}}$. 
\end{defi}

\noindent It is immediate that $\overline{\ZF}$ is classically equivalent to $\ZF$. 

\medskip

We conclude this section by showing how $\lb \cdot \rb_\pa$ can be seen as an extension of $\lb \cdot \rb_\ba$ to non-classical set theory.

\begin{lemma}\label{lemma: value of u = v are same under PA and BA}
For any complete  Boolean algebra $\mathbb{B}$ and $u, v \in \VA{\mathbb{B}}$, 
\begin{enumerate}
    \item[(i)] $\lb u = v \rb_\pa = \lb u = v \rb_\ba$ and
    \item[(ii)] $\lb u \in v \rb_\pa = \lb u \in v \rb_\ba$.
\end{enumerate}
\end{lemma}

\begin{proof}
(i) The proof will be done using the meta-induction principle on sets. Let us choose an arbitrary element $u \in \VA{\mathbb{B}}$ and assume that for any $x \in \dom(u)$, $\lb x = v \rb_\pa = \lb x = v \rb_\ba$, for all $v \in \VA{\mathbb{B}}$. Now, we prove that  $\lb u = v \rb_\pa = \lb u = v \rb_\ba$.
\allowdisplaybreaks
\begin{align*}
    \lb u = v \rb_\pa & = \bigwedge_{x \in \dom(u)} \big{(}(u(x) \Rightarrow  \lb x \in v \rb_\pa) \wedge (\lb  x \in v\rb_\pa^* \Rightarrow  u(x)^*)\big{)}\\
    & \qquad \qquad \wedge \bigwedge_{y \in \dom(v)} \big{(}(v(y) \Rightarrow  \lb y \in u \rb_\pa) \wedge ( \lb y \in  u \rb_\pa^* \Rightarrow  v(y)^*  )\big{)}\\
    & = \bigwedge_{x \in \dom(u)} \big{(}u(x) \Rightarrow  \lb x \in v \rb_\pa \big{)} \wedge \bigwedge_{y \in \dom(v)} \big{(}v(y) \Rightarrow  \lb y \in u \rb_\pa\big{)},\\
    & \qquad \qquad \qquad \mbox{(since for any two elements $a$ and $b$ of $\mathbb{B}$, $a \Rightarrow b = b^* \Rightarrow a^*$)}\\
    & = \bigwedge_{x \in \dom(u)} \big{(}u(x) \Rightarrow \bigvee_{y \in \dom(v)}(v(y) \wedge \lb x = y \rb_\pa)\big{)}\\
    & \qquad \qquad \wedge \bigwedge_{y \in \dom(v)} \big{(}v(y) \Rightarrow \bigvee_{x \in \dom(u)}(u(x) \wedge \lb x = y \rb_\pa)\big{)},\\
    & = \bigwedge_{x \in \dom(u)} \big{(}u(x) \Rightarrow \bigvee_{y \in \dom(v)}(v(y) \wedge \lb x = y \rb_\ba)\big{)}\\
    & \qquad \qquad \wedge \bigwedge_{y \in \dom(v)} \big{(}v(y) \Rightarrow \bigvee_{x \in \dom(u)}(u(x) \wedge \lb x = y \rb_\ba)\big{)},\\
    & \qquad \qquad \qquad \mbox{ (by the induction hypothesis)}\\
    & = \bigwedge_{x \in \dom(u)} \big{(}u(x) \Rightarrow  \lb x \in v \rb_\ba \big{)} \wedge \bigwedge_{y \in \dom(v)} \big{(}v(y) \Rightarrow  \lb y \in u \rb_\ba\big{)},\\
    & = \lb u = v \rb_\ba
\end{align*}
Hence, by induction we can conclude that for any $u, v \in \VA{\mathbb{B}}, ~\lb u = v \rb_\pa = \lb u = v \rb_\ba$.
(ii) Let us take any $u, v \in \VA{\mathbb{B}}$. Then,
\begin{align*}
    \lb u \in v \rb_\pa & = \bigvee_{y \in \dom(v)}(v(y) \wedge \lb y = u \rb_\pa)\\
    & = \bigvee_{y \in \dom(v)}(v(y) \wedge \lb y = u \rb_\ba), \mbox{ by using (i)}\\
    & = \lb u \in v \rb_\ba
\end{align*}
This completes the proof.
\end{proof}

\begin{teo}\label{theorem: importance of PA-assignment}
For any complete  Boolean algebra $\mathbb{B}$ and any sentence $\varphi$ of $\mathcal{L}_\mathbb{B}$, $\lb \varphi \rb_\ba = \one$ iff $\lb \varphi \rb_\pa = \one$, i.e., $\VA{\mathbb{B}, ~\lb \cdot \rb_\ba} \models \varphi$ iff $\VA{\mathbb{B}, ~\lb \cdot \rb_\pa} \models \varphi$.
\end{teo}

\begin{proof}
The proof easily follows by  induction on the complexity of the formula $\varphi$. The base case is proved in Lemma \ref{lemma: value of u = v are same under PA and BA}.
\end{proof}

\section{\textsf{Cobounded}-algebra-valued models under $\lb \cdot \rb_\ba$ and $\lb \cdot \rb_\pa$}\label{section: paraconsistent association function}

In this section, we shall explore the validity of the axioms of set theory in the designated $\co$-algebra-valued models under both the assignment functions $\lb \cdot \rb_\ba$ and $\lb \cdot \rb_\pa$. For this, we first discuss a particular $\co$-algebra $\s$ and its corresponding algebra-valued models (under $\lb \cdot \rb_\ba$-assignment function), which was developed in \cite{Loewe}.

\subsection{The algebra $\s$}\label{section: the algebra PS3}

We first discuss a three-valued algebra $\mathbb{PS}_3 = \langle \{1, \nicefrac 12, 0\}, \wedge, \vee, \Rightarrow \rangle$
having the following operations with the designated set $D = \{1, \nicefrac{1}{2}\}$:

\medskip

\begin{center}
\begin{tabular}{|c|ccc|}
  \hline
  \;$\land$\; &\; 1\; &\; $\nicefrac{1}{2}$\; &\; 0\; \\
  \hline
  \;1\; &\; 1\; &\; $\nicefrac 12$\; &\; 0\; \\
  \;$\nicefrac{1}{2}$\; &\; $\nicefrac 12$\; &\; $\nicefrac 12$\; &\; 0\; \\
  \;0\; &\; 0\; &\; 0\; &\; 0\; \\
  \hline
\end{tabular}
\hspace*{1 cm}
\begin{tabular}{|c|ccc|}
  \hline
  \;$\lor$\; &\; 1\; &\; $\nicefrac{1}{2}$\; &\; 0\; \\
  \hline
  \;1\; &\; 1\; &\; 1\; &\; 1\; \\
  \;$\nicefrac{1}{2}$\; &\; 1\; &\; $\nicefrac 12$\; &\; $\nicefrac 12$\; \\
  \;0\; &\; 1\; &\; $\nicefrac 12$\; &\; 0 \\
  \hline
\end{tabular}
\hspace*{1 cm}
\begin{tabular}{|c|ccc|}
  \hline
  \;$\Rightarrow$\; &\; 1\; &\; $\nicefrac{1}{2}$ \;&\; 0\; \\
  \hline
  \;1\; &\; 1\; &\; 1\; &\; 0\; \\
  \;$\nicefrac{1}{2}$\; &\; 1\; &\; 1\; &\; 0\; \\
  \;0 &\; 1\; &\; 1\; &\; 1\;\\
  \hline
\end{tabular}
\end{center}

\medskip

In \cite[Theorem 4.5]{Loewe} we find the proofs that $\s$ is a deductive reasonable implication algebra and that $\bq$ holds for all negation-free formula $\varphi$. This leads to the proof of $\VA{\s, \; \lb \cdot \rb_\ba} \models_D \nff$-$\ZF$ \cite[Corollary 5.2]{Loewe}. 

It is easy to check that the algebra $\s$ is a $\co$-algebra as well and, moreover, if we extend $\s$ by adding a unary operator $^*$ defined as:
\[1^* = 1, \; \nicefrac{1}{2}^* = \nicefrac{1}{2}, \; 0^* = 0,\]
the resulting structure, $\ps$ becomes an ultra-designated $\co$-algebra, for the same designated set $D = \{1, \nicefrac{1}{2}\}$.

\subsection{$\co$-algebra-valued models under $\lb \cdot \rb_\ba$}

We have already notices that  the two-element $\co$-algebra coincides with  the two-element Boolean algebra. Hence, the two-element $\co$-algebra-valued model is nothing but the two-element Boolean-valued model of $\ZFC$. 

From now on for a given $\co$-algebra $\A = \langle \mathbf{A}, \wedge, \vee, \Rightarrow, \one, \zero \rangle$, we shall denote a function from $\A$ into $\s$ as $f_\A$, which will be defined by:
\begin{align*}
   f_\A(a) & = \left \{\begin{array}{ll}
                        1, & \mbox{if } a = \one;  \\
                        \nicefrac{1}{2}, & \mbox{if } a \in \mathbf{A} \setminus \{\one, \zero\};\\
                        0, & \mbox{if } a = \zero.
                        \end{array}\right.
\end{align*}

\begin{lemma}\label{lemma: f_A is a homomorphism}
For any $\co$-algebra $\A$, $f_\A$ is a homomorphism.
\end{lemma}

\begin{proof}
Let $a, b \in \mathbf{A}$ be any two elements. To prove that $f_\A$ is a homomorphism, we need to prove the following facts.
\begin{enumerate}
    \item[(i)] $f_\A(a \wedge b) = f_\A(a) \wedge f_\A(b)$.\\
    If both of $a$ and $b$ are $\one$ or at least one of them is $\zero$, the proof is immediate. Suppose $a = \one$ and $b \in \mathbf{A} \setminus \{\one, \zero\}$. Then,
    \[f_\A(a \wedge b) = f_\A(b) = \nicefrac{1}{2} = 1 \wedge \nicefrac{1}{2} = f_\A(a) \wedge f_\A(b).\]
    \item[(ii)] $f_\A(a \vee b) = f_\A(a) \vee f_\A(b)$.\\
    The proof is similar as Case I.
    \item[(iii)]  $f_\A(a \Rightarrow b) = f_\A(a) \Rightarrow f_\A(b)$.\\
    Note that, both of $f_\A(a \Rightarrow b)$ and $f_\A(a) \Rightarrow f_\A(b)$ will either be 1 or 0. It is sufficient to prove that $f_\A(a \Rightarrow b) = 0$ iff $f_\A(a) \Rightarrow f_\A(b) = 0$. We have,
    \begin{align*}
        f_\A(a \Rightarrow b) = 0 & \mbox{ iff } b = \zero \mbox{ and } a \neq \zero\\
        & \mbox{ iff } f_\A(b) = 0 \mbox{ and } f_\A(a) \neq 0\\ 
        & \mbox{ iff } f_\A(a) \Rightarrow f_\A(b) = 0.
    \end{align*}
\end{enumerate}
\end{proof}

\begin{lemma}\label{lemma: completeness of the homomorphism}
For any $\co$-algebra $\A$ and a subset $\{a_i : i \in I\}$ of the domain of $\A$, $f_\A \big{(} \bigwedge_{i \in I} a_i \big{)} = \bigwedge_{i \in I}  \big{(} f_\A(a_i) \big{)}$ and $f_\A \big{(} \bigvee_{i \in I} a_i \big{)} = \bigvee_{i \in I}  \big{(} f_\A(a_i) \big{)}$, where $I$ is an index set.
\end{lemma}

\begin{proof}
We first prove $f_\A \big{(} \bigwedge_{i \in I} a_i \big{)} = \bigwedge_{i \in I}  \big{(} f_\A(a_i) \big{)}$. If $\bigwedge_{i \in I} a_i = \one$, then all $a_i = \one$. This implies $f_\A(a_i) = \one$, for all $i \in I$. Hence, 
\[f_\A \big{(} \bigwedge_{i \in I} a_i \big{)} = f_\A(\one) = 1 = \bigwedge_{i \in I}  \big{(} f_\A(a_i) \big{)}.\]
If $\bigwedge_{i \in I} a_i \in \mathbf{A} \setminus \{\one, \zero\}$, then there exists $j \in I$ so that $f_\A(a_j) \in \mathbf{A} \setminus \{\one, \zero\}$ and $f_\A(a_i) \in \{1, \nicefrac{1}{2}\}$, for all $i \in I$. Hence, 
\[f_\A \big{(} \bigwedge_{i \in I} a_i \big{)} = \nicefrac{1}{2} = \bigwedge_{i \in I}  \big{(} f_\A(a_i) \big{)}.\]
If $\bigwedge_{i \in I} a_i = \zero$, then by the definition of $\co$-algebra, there exists $j \in I$ such that $a_j = \zero$. So $f_\A(a_j) = 0$. Hence,
\[f_\A \big{(} \bigwedge_{i \in I} a_i \big{)} = f_\A(\zero) = 0 = \bigwedge_{i \in I}  \big{(} f_\A(a_i) \big{)}.\]
To prove $f_\A \big{(} \bigvee_{i \in I} a_i \big{)} = \bigvee_{i \in I}  \big{(} f_\A(a_i) \big{)}$, we can use the similar arguments and the definition of $\co$-algebra.
\end{proof}

\begin{defi}
For a given $\co$-algebra $\A$ and an element $u \in \VA{\A}$ an element $\bar{u} \in \VA{\s}$ is defined as follows:
\[\bar{u} := \{\langle \bar{x}, f_\A(a) \rangle : \langle x , a \rangle \in u\}.\]
\end{defi}
Intuitively, $\bar{u}$ is formed by replacing all the entries of $\one$ and $\zero$ in the formation of $u$ by 1 and 0, respectively; 
and similarly by replacing all the occurrences of any element $a \in \mathbf{A} \setminus \{\one, \zero\}$ in $u$ by $\nicefrac{1}{2}$. 
Hence, for any $u \in \VA{\A}$ there exists a unique $\bar{u} \in \VA{\s}$ and conversely, for any $\bar{u} \in \VA{\s}$ there exists at least one $u \in \VA{\A}$ (more than one if $\A$ contains more than three elements).

\begin{teo}\label{theorem: cobounded equality}
Let $\A$ be a $\co$-algebra. Then, $f_\A(\lb u = v \rb_\ba) = \lb \bar{u} = \bar{v} \rb_\ba$, for all $u, v \in \VA{\A}$.
\end{teo}

\begin{proof}
We shall prove the theorem using the (meta) induction. Let the theorem hold for all elements $x \in \dom(u)$, i.e., $f_\A(\lb x = v \rb_\ba) = \lb \bar{x} = \bar{v} \rb_\ba$, for all $v \in \VA{\A}$. It is now sufficient to prove that $f_\A(\lb u = v \rb_\ba) = \lb \bar{u} = \bar{v} \rb_\ba$, for all $v \in \VA{\A}$. Using Lemma \ref{lemma: completeness of the homomorphism} in the necessary steps we have,
\allowdisplaybreaks
{
\begin{align*}
    f_\A(\lb u = v \rb_\ba) & = f_\A \big{(}  \bigwedge_{x \in \mathrm{dom}(u)} (u(x) \Rightarrow \lb x \in v \rb_\ba) \land \bigwedge_{y \in \mathrm{dom}(v)} (v(y) \Rightarrow \lb y \in u \rb_\ba ) \big{)}\\
    & = \bigwedge_{x \in \mathrm{dom}(u)} \big{(} f_\A(u(x)) \Rightarrow f_\A (\lb x \in v \rb_\ba)  \big{)} \land \bigwedge_{y \in \mathrm{dom}(v)} \big{(} f_\A(v(y)) \Rightarrow f_\A(\lb y \in u \rb_\ba ) \big{)}\\
    & = \bigwedge_{x \in \mathrm{dom}(u)} \big{(} f_\A(u(x)) \Rightarrow \bigvee_{y \in \dom(v)} \big{(}f_\A(v(y)) \wedge f_\A(\lb x = y \rb_\ba) \big{)} \big{)} \land \\
    & \hspace{2 cm}\bigwedge_{y \in \mathrm{dom}(v)} \big{(} f_\A(v(y)) \Rightarrow \bigvee_{x \in \dom(u)} \big{(}f_\A(u(x)) \wedge f_\A(\lb x = y \rb_\ba) \big{)} \big{)}\\
    & = \bigwedge_{\bar{x} \in \mathrm{dom}(\bar{u})} \big{(} \bar{u}(\bar{x}) \Rightarrow \bigvee_{\bar{y} \in \dom(\bar{v})} \big{(} \bar{v}(\bar{y}) \wedge (\lb \bar{x} = \bar{y} \rb_\ba) \big{)} \big{)} \land \\
    & \hspace{2 cm} \bigwedge_{\bar{y} \in \mathrm{dom}(\bar{v})} \big{(} \bar{v}(\bar{y}) \Rightarrow \bigvee_{\bar{x} \in \dom(\bar{u})} \big{(} \bar{u}(\bar{x}) \wedge (\lb \bar{x} = \bar{y} \rb_\ba) \big{)} \big{)}\\
    & = \bigwedge_{\bar{x} \in \mathrm{dom}(\bar{u})} \big{(} \bar{u}(\bar{x}) \Rightarrow \lb \bar{x} \in \bar{v} \rb_\ba \big{)} \land \bigwedge_{\bar{y} \in \mathrm{dom}(\bar{v})} \big{(} \bar{v}(\bar{y}) \Rightarrow \lb \bar{y} \in \bar{u} \rb_\ba \big{)}\\
    & = \lb \bar{u} = \bar{v} \rb_\ba.
\end{align*}
}
Hence, applying (meta) induction the theorem is proved.
\end{proof}

\begin{teo}\label{theorem: cobounded belongingness}
Let $\A$ be a $\co$-algebra. Then, $f_\A(\lb u \in v \rb_\ba) = \lb \bar{u} \in \bar{v} \rb_\ba$, for all $u, v \in \VA{\A}$.
\end{teo}

\begin{proof}
We have the following:
\begin{align*}
    f_\A(\lb u \in v \rb_\ba) & = f_\A \big{(} \bigvee_{y \in \dom(v)} (v(y) \wedge \lb y = u \rb_\ba) \big{)}\\
    & = \bigvee_{y \in \dom(v)} \big{(} f_\A(v(y)) \wedge f_\A(\lb y = u \rb_\ba) \big{)}\\
    & = \bigvee_{\bar{y} \in \dom(\bar{v})} \big{(} \bar{v}(\bar{y}) \wedge \lb \bar{y} = \bar{u} \rb_\ba \big{)}, \mbox{ by Theorem \ref{theorem: cobounded equality}}\\
    & = \lb \bar{u} \in \bar{v} \rb_\ba.
\end{align*}
\end{proof}

In simple words,
Theorems \ref{theorem: cobounded equality} and \ref{theorem: cobounded belongingness} states that for any two elements $u, v \in \VA{\A}$, 
\begin{enumerate}
    \item[(i)] if $\lb u = v \rb_\ba = \one$, then $\lb \bar{u} = \bar{v} \rb_\ba = 1$,
    \item[(ii)] if $\lb u = v \rb_\ba = \zero$, then $\lb \bar{u} = \bar{v} \rb_\ba = 0$,
    \item[(iii)] if $\lb u \in v \rb_\ba = \one$, then $\lb \bar{u} \in \bar{v} \rb_\ba = 1$,
    \item[(iv)] if $\lb u \in v \rb_\ba \in \mathbf{A} \setminus \{\one, \zero\}$, then $\lb \bar{u} \in \bar{v} \rb_\ba = \nicefrac{1}{2}$,
    \item[(v)] if $\lb u \in v \rb_\ba = \zero$, then $\lb \bar{u} \in \bar{v} \rb_\ba = 0$.
\end{enumerate}
This result will be extended to any negation-free sentence of the extended language $\mathcal{L}_\A$. To simplify  notation we can use the expressions $\lb \varphi \rb_\ba^\A$ and $\lb \varphi \rb_\ba^{\s}$ to indicate the value of the sentence $\varphi$ under $\lb \cdot \rb_\ba$-assignment function in $\VA{\A}$ and $\VA{\s}$, respectively.

\begin{teo}\label{theorem: value in A and PS3 are same}
Let $\A$ be a $\co$-algebra. For any negation-free sentence $\varphi \in \mathcal{L}_\A$, $f_\A(\lb \varphi \rb_\ba^\A) = \lb \varphi \rb_\ba^{\s}$.
\end{teo}

\begin{proof}
The theorem will be proved using mathematical induction on the complexity of the negation-free sentence $\varphi$ of $\mathcal{L}_\A$.\\
The Base Case is proved in Theorem \ref{theorem: cobounded equality} and Theorem \ref{theorem: cobounded belongingness}.\\
Suppose the theorem holds for any sentence having complexity $n$, where $n$ is a natural number.\\
Let $\varphi$ be a sentence having complexity $n+1$. Then, the following cases may occur.
\begin{enumerate}
    \item[(i)] Case I: Let $\varphi := \psi \wedge \gamma$. Then,
    \[\lb \varphi \rb_\ba^\A = \lb \psi \rb_\ba^\A \wedge \lb \gamma \rb_\ba^\A = \lb \psi \rb_\ba^{\s} \wedge \lb \gamma \rb_\ba^{\s} \mbox{ (by induction hypothesis) } = \lb \varphi \rb_\ba^{\s}.\]
    \item[(ii)] Case II: Let $\varphi := \psi \vee \gamma$. The proof is similar to the proof of Case I.
    \item[(iii)] Case III: Let $\varphi := \psi \Rightarrow \gamma$. This proof is also similar to Case I.
    \item[(iv)] Case IV: Let $\varphi := \forall x \psi(x)$. Then,
    \begin{align*}
        f_\A(\lb \varphi \rb_\ba^\A) & = f_\A \big{(} \bigwedge_{u \in \VA{\A}}\lb \psi(u) \rb_\ba^{\A} \big{)}\\
        & = \bigwedge_{u \in \VA{\A}} f_\A \big{(}\lb \psi(u) \rb_\ba^{\A} \big{)}, \mbox{ by Lemma \ref{lemma: completeness of the homomorphism}}\\
        & = \bigwedge_{\bar{u} \in \VA{\s}} f_\A \big{(}\lb \psi(\bar{u}) \rb_\ba^{\s} \big{)}, \mbox{ by Lemma \ref{lemma: completeness of the homomorphism}}\\
        & = f_\A \big{(} \bigwedge_{\bar{u} \in \VA{\s}}\lb \psi(\bar{u}) \rb_\ba^{\s} \big{)}\\
        & = f_\A(\lb \varphi \rb_\ba^{\s})
    \end{align*}
    \item[(v)] Case V: Let $\varphi := \exists x \psi(x)$. The proof is same as the proof of Case IV.
\end{enumerate}
Hence, by using mathematical induction we can conclude that the theorem holds for all negation-free sentences of $\mathcal{L}_\A$.
\end{proof}

\begin{teo}
Let $\A$ be a $\co$-algebra and $D$ be any designated set of it. Then, $\VA{\A, \; \lb \cdot \rb_\ba} \models_D \nff$-$\ZF$.
\end{teo}

\begin{proof}
Let $\mathsf{Ax}$ be an axiom of $\nff$-$\ZF$. Then, it is of the form $Q_1x_1 \ldots Q_nx_n (\varphi \to \psi)$, where $Q_1, \ldots, Q_n$ are quantifiers. 
Therefore, $\lb \mathsf{Ax} \rb_\ba^\A$ is either $\one$ or $\zero$. 
If $\lb \mathsf{Ax} \rb_\ba^\A = \zero$, then $f_\A(\lb \mathsf{Ax} \rb_\ba^\A) = 0$ and hence by Theorem \ref{theorem: value in A and PS3 are same}, $\lb \mathsf{Ax} \rb_\ba^{\s} = 0$. 
This leads to a contradiction since $\VA{\s, \; \lb \cdot \rb_\ba} \models_D \nff$-$\ZF$. 
Hence, $\lb \mathsf{Ax} \rb_\ba^\A = \one$ and therefore for any designated set $D$, $\lb \mathsf{Ax} \rb_\ba^\A \in D$. This completes the proof.
\end{proof}

Let us now consider any designated $\co$-algebra $(\A, D)$, having more than two elements. Then by definition, there is a unary operator $^*$ in the structure of $\A$. One may ask whether $\VA{\A, \; \lb \cdot \rb_\ba}$ validates all $\ZF$-axioms instead of only its negation-free fragment: $\nff$-$\ZF$. 

Consider a designated $\co$-algebra $(\A, D)$,
where $\mathbf{A}$ is an ordinal number $\geq 3$. Notice that, $\mathbf{A}$ can neither be a limit ordinal (because $\A$ needs to have a top element) nor a successor of a limit ordinal (because $\A$ needs to haev a coatom). Then, one can be prove that there exists a formula $\varphi(x):= \lnot \exists y (y \in x)$ such that $\VA{\A, \; \lb \cdot \rb_\ba} \not\models_D \mathsf{Separation}_\varphi$ \cite[Theorem 59]{Taraf}. Hence in general, a designated  $\co$-algebra-valued model cannot validate the full $\ZF$ under the $\lb \cdot \rb_\ba$-assignment function.

\subsection{Set-theoretic properties of $\mathsf{Cobounded}$-algebra-valued models under $\lb \cdot \rb_\pa$}

The $\lb \cdot \rb_\pa$ assignment function considerably affects  the way we look at the elements of $\mathsf{Cobounded}$-algebra-valued models. First of all, notice that for any designated \textsf{Cobounded}-algebra $(\mathbb{A},D)$ and any $u, v \in \VA{\A}$,  either $\lb u = v \rb_\pa = \one$ or $\lb u = v \rb_\pa = \zero$.

\begin{teo}\label{theorem: equality corresponding to paraconsistent assignment function}
Consider any ultra-designated $\mathsf{Cobounded}$-algebra $(\mathbb{A},D)$ and any two elements $u, v \in \VA{\A}$.  Then, $\VA{\A, ~ \lb \cdot \rb_\pa} \models_D u = v$ if and only if the following holds: if $x \in \dom(u)$ is such that $u(x) \in D$ then there exists $y \in \dom(v)$ such that $v(y) \in D$ and $\lb x = y \rb_\pa \in D$ and vice-versa; moreover, if $u(x) = \one$, then there exists $y \in \dom(v)$ such that $v(y) = \one$ and $\lb x = y \rb_\pa \in D$ and vice-versa.
\end{teo}

\begin{proof}
Since $D$ is an ultrafilter of $\A$, by  Lemma \ref{lemma: ultrafilter of an MTV algebra} we have that  $\mathbf{A} \setminus D = \{\zero\}$, where $\mathbf{A}$ is the underlying set of $\A$. Let us consider two elements $u, v \in \VA{\A}$ such that $\VA{\A, ~ \lb \cdot \rb_\pa} \models_D u = v$.  Then, the following two cases, Case I and Case II are sufficient to prove the property of the theorem.

Case I: Suppose there exists an element $x \in \dom(u)$ such that $u(x)= \one$. Then, it is clear that $u(x)^* = \zero$. If there is no $y \in \dom(v)$ such that $v(y)=\one$ and $\lb x = y \rb_\pa \in D$,  we get $\lb x \in v \rb_\pa < \one$ and hence $\lb x \in v \rb_\pa^* \in D$. 
This concludes that, $\big{(}\lb x \in v \rb_\pa^* \Rightarrow u(x)^* \big{)} = \zero$, which contradicts the assumption. 
Similarly, if there exists $y \in \dom(v)$ such that $v(y)= \one$, then there also exists $x \in \dom(u)$ such that $u(x)=\one$ and   $\lb x = y \rb_\pa \in D$, otherwise $\big{(}\lb y \in u \rb_\pa^* \Rightarrow v(y)^* \big{)} = \zero$, and hence our assumption fails.

Case II: Suppose there exists an element $x \in \dom(u)$ such that $u(x) \in D \setminus \{\one\}$. Towards the contradiction, suppose there does not exist any $y \in \dom(v)$ such that $v(y) \in D$ and  $\lb x = y \rb_\pa \in D$, i.e., for every $y \in \dom(v)$ either $v(y) = \zero$ or $\lb x = y \rb_\pa = \zero$. Hence, $\big{(}u(x) \Rightarrow \lb x \in v \rb_\pa \big{)} = \zero$, which implies that $\lb u = v \rb_\pa = \zero$.

Conversely, let the condition hold.  
Consider an arbitrary $x \in \dom(u)$. If $u(x) = \zero$, then immediately we get,
\[(u(x) \Rightarrow  \lb x \in v \rb_\pa) \wedge ( \lb  x \in v\rb_\pa^* \Rightarrow  u(x)^*) \in D.\]
Now, suppose $u(x) \in D$. Then, by our assumption $\lb x \in v \rb_\pa \in D$. Hence, the first conjunct of the algebraic expression of $\lb u = v \rb_\pa$, $(u(x) \Rightarrow  \lb x \in v \rb_\pa) \in D$.
For the second conjunct, we consider two cases: $u(x) \in D \setminus \{\one\}$ and $u(x) = \one$.
If $u(x) \in D \setminus \{\one\}$, then $u(x)^* \in D$ and $\lb x \in v \rb_\pa \in D$ by our assumption. On the other hand, if $u(x) = \one$, then by our assumption, $\lb x \in v \rb_\pa = \one$, so $\lb x \in v \rb_\pa^* = \zero$.  Hence, in both the cases we have
\[(u(x) \Rightarrow  \lb x \in v \rb_\pa) \wedge ( \lb  x \in v\rb_\pa^* \Rightarrow  u(x)^*) \in D.\]
A similar argument shows that, for any $y \in \dom(v)$,
\[(v(y) \Rightarrow  \lb y \in u \rb_\pa) \wedge ( \lb y \in  u \rb_\pa^* \Rightarrow  v(y)^*) \in D.\]
This leads to the fact that, $\VA{\mathbb{A}, ~ \lb \cdot \rb_\pa} \models_D u = v$.
\end{proof}

We now prove in Observation \ref{obs: difference in equality} and Observation \ref{obs: difference between belongingness} that the assignment function $\lb \cdot \rb_\pa$ is finer than the assignment function $\lb \cdot \rb_\ba$ in validating the equality and set-membership relation.

\begin{observation}\label{obs: difference in equality}
Let $(\mathbb{A},D)$ be an ultra-designated $\mathsf{Cobounded}$-algebra and $u, v \in \VA{\A}$ be any two elements. If  $\VA{\A, ~ \lb \cdot \rb_\pa} \models_D u = v$, then  $\VA{\A, ~ \lb \cdot \rb_\ba} \models_D u = v$. But the converse may not be true.
\end{observation}

\begin{proof}
It is not hard to check that, for any $u, v \in \VA{\A}$, $\lb u = v \rb_\ba \in D$ if and only if the following hold:
if $u(x) \in D$ then there exists $v(y) \in  D$ such that $\lb x = y \rb_\ba \in D$, and vice-versa.
Then, applying Theorem \ref{theorem: equality corresponding to paraconsistent assignment function} we can conclude that, if  $\VA{\A, ~ \lb \cdot \rb_\pa} \models_D u = v$, then  $\VA{\A, ~ \lb \cdot \rb_\ba} \models_D u = v$.

The converse does not always hold:  fix two elements $u, v \in \VA{\A}$ such that, $u=\{\langle\ w, \one \rangle\}$, $v=\{\langle w,a\rangle\}$, where $w \in \VA{\A}$ is an arbitrary element and $a \in \A$ is such that $\zero < a < \one$. 
Then, $\VA{\A, ~ \lb \cdot \rb_\ba} \models_D u = v$, but Theorem \ref{theorem: equality corresponding to paraconsistent assignment function} ensures that, $\VA{\A, ~ \lb \cdot \rb_\pa} \not\models_D u = v$.
\end{proof}

 \begin{observation}\label{obs: difference between belongingness}
 Let $(\mathbb{A},D)$ be an ultra-designated $\mathsf{Cobounded}$-algebra and $u, v \in \VA{\A}$ be any two elements. If  $\VA{\A, ~ \lb \cdot \rb_\pa} \models_D u \in v$, then  $\VA{\A, ~ \lb \cdot \rb_\ba} \models_D u \in v$. But the converse may not be true. 

\end{observation}

\begin{proof}
Let $\VA{\A, ~ \lb \cdot \rb_\pa} \models_D u \in v$ holds, for two elements $u, v \in \VA{\A}$. Then, there exists $t \in \dom(v)$ such that $v(t) \in D$ and $\lb u = t \rb_\pa \in D$. By Observation \ref{obs: difference in equality}, it can be concluded that $\lb u = t \rb_\ba \in D$ as well. Hence we have $\VA{\A, ~ \lb \cdot \rb_\ba} \models_D u \in v$.

 The converse is not always true, as the following example shows.
 Let us first consider two elements $u=\{\langle w, \one \rangle\}$ and $t =\{\langle w,a\rangle\}$, where $w \in \VA{\A}$ and $a \in \A$ is an arbitrary element such that  $\zero < a < \one$. Then, take the element
 $v = \{\langle t, \one \rangle\}$ in $\VA{\A}$. As  already observed in the proof of Observation \ref{obs: difference in equality} we have that $\lb u = t \rb_\ba \in D$, but $\lb u = t \rb_\pa \notin D$. Hence we get, $\VA{\A, ~ \lb \cdot \rb_\ba} \models_D u \in v$, however $\VA{\A, ~ \lb \cdot \rb_\pa} \not\models_D u \in v$.
\end{proof}

The following set-theoretic properties displayed in Lemma \ref{PropertiesTvaluedmodels},  hold for both,  Boolean and Heyting-valued models, equipped with the assignment function $\lb \cdot \rb_\ba$.

\begin{lemma} \label{PropertiesTvaluedmodels}   
Let $(\A,D)$ be an ultra-designated $\mathsf{Cobounded}$-algebra.  Then,
for any  $u, v, w \in \V^{\mathbb{(A)}}$ the following hold:

\begin{enumerate}[label=(\roman*)]

\item $\lb u=u \rb_\pa \in D $,

\item for any $x \in \dom(u)$, if $ u(x) \in D$, then $\lb x \in u \rb_\pa \in D$, 

\item if $\lb u=v \rb_\pa \wedge \lb v = w \rb_\pa \in D$, then $\lb u=w \rb_\pa \in D$, 
    
\item if $\lb u=v \rb_\pa \wedge \lb v \in w \rb_\pa \in D$, then $\lb u \in w \rb_\pa \in D$,
    
\item if $\lb u=v \rb_\pa \wedge \lb w \in v \rb_\pa \in D$, then $\lb w \in u \rb_\pa \in D$.
\end{enumerate}
\end{lemma}

\begin{proof} $(i)$ Consider any $x \in \dom(u)$ such that $u(x) \in D$. Let us consider the following two cases: $u(x)= \one$ and $u(x) \in D$.  
In the first case, by Theorem \ref{theorem: equality corresponding to paraconsistent assignment function} we get $\lb x \in u  \rb_{\pa} = \one$. 
Thus:\[(u(x) \Rightarrow  \lb x \in u \rb_\pa) \wedge ( \lb  x \in u\rb_\pa^* \Rightarrow  u(x)^*) \in D.\] 
In the second case, since $u(x) \in D \setminus \{\one\}$ we have $u(x)^* \in D \setminus \{\one\}$ and by Theorem \ref{theorem: equality corresponding to paraconsistent assignment function} we get $\lb x \in u  \rb_{\pa} \in D$. It follows immediately that: \[(u(x) \Rightarrow  \lb x \in u \rb_\pa) \wedge ( \lb  x \in u\rb_\pa^* \Rightarrow  u(x)^*) \in D.\] We may conclude $\lb u=u \rb_{\pa} \in D$ for any $u \in \V^{(\mathbb{A})}$.

\medskip

\noindent $(ii)$  Let $u(x) \in D$, then we have \[\llbracket x \in u \rrbracket_\pa \geq (u(x) \wedge  \llbracket x = x \rrbracket_\pa) \in D,\] since $\llbracket x = x \rrbracket_\pa \in D$ by item $(i)$.

\medskip

\noindent $(iii)$ By induction on the domain of $w$. Assume that for all $z \in \dom(w)$ we have:
\[\llbracket u=v \rb_\pa \wedge \llbracket v=z \rb_\pa \in D \mbox{ implies }  \llbracket u = z \rb_\pa \in D.\] Case I: Let  $x_1,..., x_n \in \dom(u)$ such that  $u(x_i)= \one$, $i= 1,...,n$. Then by Theorem \ref{theorem: equality corresponding to paraconsistent assignment function}, there exist $y_1,...,y_n \in \dom(v)$ such that $v(y_i)= \one$ and  $\lb x_i = y_i \rb_\pa \in D$,  $i= 1,...,n$. Since, $\lb v=w \rb_\pa \in D$, as well, there exist $z_1,...,z_n \in \dom(w)$ such that $w(z_i)= \one$ and  $\lb y_i = z_i \rb_\pa\in D$,  $i= 1,...,n$. Now, by induction hypothesis $\lb x_i = z_i \rb_\pa \in D$, $i= 1,...,n$.
Case II: Let there exist $x_1^{\prime},..., x_k^{\prime} \in \dom(u)$ such that  $u(x_i^{\prime}) \in D \setminus \{\one\}$, $i= 1,...,k$. Then by Theorem \ref{theorem: equality corresponding to paraconsistent assignment function}, there exist $y_1^{\prime},...,y_k^{\prime} \in \dom(v)$ such that $v(y_i^{\prime})\in D \setminus \{\one\}$ 
and  $\lb x_i^{\prime} = y_i^{\prime} \rb_\pa \in D$,  $i= 1,...,n$.  Since, $\lb v=w \rb_\pa \in D$, as well, there exist $z_1^{\prime},...,z_k^{\prime} \in \dom(w)$ such that $w(z_i^{\prime})\in D \setminus \{\one\}$ 
and  $\lb y_i^{\prime} = z_i^{\prime} \rb_\pa \in D$, and moreover if $v(y_i^{\prime})= \one$, then $w( z_i^{\prime})= \one$, $i= 1,...,n$. 
By induction hypothesis,  $\lb x_i^{\prime}  = z_i^{\prime}  \rb_\pa \in D$, $i= 1,...,n$.  We can prove the vice-versa similarly. 
\medskip

\noindent $(iv)$ Assume that $\lb u = v \rb_\pa \wedge \lb u \in w \rb_\pa \in D$. 
\begin{align*}\lb u=v \rb_\pa \wedge \llbracket u \in w \rb_\pa &= 
\lb u=v \rb_\pa \wedge  \bigvee_{z \in \dom(w)} (w(z) \wedge \llbracket z=u \rrbracket_\pa)\\
&= \bigvee_{z \in \dom(w)} w(z) \wedge (\llbracket z=u \rrbracket_\pa \wedge \llbracket u = v \rb_\pa)
\end{align*}
Hence, $\lb z=u \rb_\pa \wedge \lb u = v \rb_\pa \in D$ which implies by item $(iii)$ that $\lb  z = v \rb_\pa \in D$. Therefore, \[ \bigvee_{z \in \dom(w)} \big{(} w(z) \wedge \lb  z = v \rb_\pa \big{)} = \llbracket v \in w \rrbracket_\pa \in D.\] 

\medskip

\noindent$(v)$ Suppose, both $\llbracket u=v \rrbracket_\pa \in D$ and $\llbracket w \in v \rrbracket_\pa \in D$ hold. Then, $\llbracket w \in v \rrbracket_\pa \in D$ implies that 
there exists $x_0 \in \dom(v)$ so that $v(x_0) \in D$ and $\lb x_0 = w \rb_\pa \in D$. 
Since $\lb u = v \rb_\pa \in D$, using Theorem \ref{theorem: equality corresponding to paraconsistent assignment function}, we can ensure that there exists $y_0 \in \dom(u)$ such that $u(y_0) \in D$ and $\lb x_0 = y_0 \rb_\pa \in D$. 
By applying $(iii)$ we can see that $\lb y_0 = w \rb_\pa \in D$. 
Hence, $u(y_0) \wedge \lb y_0 = w \rb_\pa \in D$, which proves that $\lb w \in u \rb_\pa \in D$.
\end{proof}

\begin{lemma}\label{lemma: classification of designated values by ass. function}
Let $(\A,D)$ be an ultra-designated $\mathsf{Cobounded}$-algebra. For any $u, v \in \VA{\A}$ and any formula $\varphi(x)$ in $\mathcal{L}_\A$ having one free variable $x$, if $\lb u = v \rb_\pa \in D$ then the following hold:
\begin{enumerate}
    \item[(i)] if $\lb \varphi(u) \rb_\pa = \one$ then $\lb \varphi(v) \rb_\pa = \one$,
    \item[(ii)] if $\zero < \lb \varphi(u) \rb_\pa < \one$ then $\zero < \lb \varphi(v) \rb_\pa < \one$.
\end{enumerate}
\end{lemma}

\begin{proof} We divide the base case into the following three cases: \textbf{BS I}, \textbf{BS II}, and \textbf{BS III}.

\medskip

\paragraph{BS I:}
Let $\varphi(x):= w = x$, where $w \in \VA{\A}$. If $\lb \varphi(u) \rb_\pa = \one$, then by Theorem \ref{PropertiesTvaluedmodels}$(iii)$ we have that $\lb \varphi(v) \rb_\pa = \one$, as for any $p, q \in \VA{\A}$, either $\lb p = q \rb_\pa = \one$ or $\lb p = q \rb_\pa = \zero$.

Condition $(ii)$ is vacuously true in this case.

\medskip

\paragraph{BS II:} Let $\varphi(x):= w \in x$, where $w \in \VA{\A}$.
Suppose $\lb \varphi(u) \rb_\pa = \one$. 
Then, there exists $p \in \dom(u)$ such that $u(p) = \one$ and $\lb p = w \rb_\pa = \one$. Since $\lb u = v \rb_\pa \in D$, by Theorem \ref{theorem: equality corresponding to paraconsistent assignment function}, there exists $q \in \dom(v)$ satisfying $v(q) = \one$ and $\lb p = q \rb_\pa \in D$. 
By Theorem \ref{PropertiesTvaluedmodels}$(iii)$, $\lb q = w \rb_\pa \in D$, i.e., $\lb q = w \rb_\pa = \one$. Hence $\lb \varphi(v) \rb_\pa = \one$.

Suppose $\zero < \lb \varphi(u) \rb_\pa < \one$, i.e., 

\[\zero<\bigvee_{x \in \dom(u)} (u(x) \wedge \lb x=w \rb_\pa )< \one. \] This can only be the case if;

\begin{enumerate}
    \item There exists $x_1 \in \dom(u)$ such that $\zero < u(x_1) <\one$ and $\lb x_1 = w \rb_\pa = \one$. 
    \item For any $x \in \dom(u)$, if $u(x)= \one$ then $\lb x=w \rb_\pa= \zero$.
\end{enumerate}
\noindent Now, \[\lb w \in v \rb= \bigvee_{y \in \dom(v)}(v(y) \wedge \lb y=w \rb_\pa).\] Since $\lb u=v \rb_\pa= \one$ and (1) holds, by Theorem \ref{theorem: equality corresponding to paraconsistent assignment function} we can say that there exists a $y_1 \in \dom(v)$ such that $v(y_1) \in D$ and $\lb y_1 =x_1 \rb_\pa = \one$. Thus by using Lemma \ref{PropertiesTvaluedmodels}(iii), $\lb w \in v \rb_\pa \in D$, i.e., $\zero<\lb w \in v \rb_\pa$. 
We shall now prove that $\lb w \in v \rb_\pa <\one$. Suppose otherwise, then by Definition \ref{Definition: MTV algebra}(iii) there exists $y_1 \in \dom(v)$ such that $v(y_1)= \one = \lb y_1=w \rb_\pa$. 
Since $\lb u=v \rb_\pa= \one$ by Theorem \ref{theorem: equality corresponding to paraconsistent assignment function} there exists a $x_1 \in \dom(u)$ such that $u(x_1)= \one = \lb x_1 = y_1 \rb_\pa$. 
By Lemma \ref{PropertiesTvaluedmodels}$(iii)$, we have also $\lb x_1 =  w \rb_\pa = \one$.  This contradicts that $0<\lb w \in u \rb_\pa< \one$. Hence we get, $\zero <\lb w \in v \rb < \one$.

\medskip

\paragraph{BS III:}
Let $\varphi(x):= x \in w$, where $w \in \VA{\A}$. Suppose $\lb \varphi(u) \rb_\pa = \one$, i.e., \[\bigvee_{z \in \dom(w)} (w(z) \wedge \lb u=z \rb_\pa ) = \one. \] So there exists a $z_1 \in \dom(w)$ such that $w(z_1)= \one = \lb z_1 = u \rb_\pa$. Thus, we have that $\lb u=v \rb_\pa \wedge \lb z_1 = u \rb_\pa \in D$ and  by Lemma \ref{PropertiesTvaluedmodels}(iii), $\lb z_1 = v \rb_\pa \in D$. So there exists a $z_1 \in \dom(w)$ such that $w(z_1)= \one = \lb z_1 =v \rb_\pa$, i.e., $\lb v \in w \rb_\pa = \one$.

Suppose $ \zero < \lb \varphi(u) \rb_\pa < \one$, i.e., 
\[\zero <  \bigvee_{z \in \dom(w)} (w(z) \wedge \lb u=z \rb_\pa) < \one. \] This can only be the case if;
\begin{enumerate}
    \item There exists $z_1 \in \dom(w)$ such that $\zero < w(z_1) <  \one $ and $\lb z_1 = u \rb_\pa = \one$. 
    \item For any $z \in \dom(w)$, if $w(z)= \one$ then $\lb z=u \rb_\pa = \zero$.
\end{enumerate}   By Lemma \ref{PropertiesTvaluedmodels}(iii) we have $\lb z_1 =v \rb_\pa \in D$. So there exists a $z_1 \in \dom(w)$ such that $\zero < w(z_1) <  \one  \mbox{ and } \lb z_1 =v \rb_\pa \in D$, i.e., $\lb v \in w \rb_\pa \in D$. We shall now prove that we have $\lb v \in w \rb_\pa <\one$. Suppose otherwise, so there exists a $z_2 \in \dom(w)$ such that $w(z_2)= \one = \lb z_2=v \rb_\pa$.
Since $\lb u=v \rb_\pa= \one$ by Lemma \ref{PropertiesTvaluedmodels}$(iii)$, we have $\lb z_2 =  u \rb_\pa = \one$. So there exists a  $z \in \dom(w)$ such that $w(z)= \one$ and $\lb z=u \rb_\pa = \one$. This contradicts that $\lb u \in w \rb_\pa < \one$. Hence we get, $\zero <\lb v \in w \rb_\pa < \one$.

\medskip
\paragraph{Induction hypothesis}
Suppose the lemma holds for any formula having the complexity less than $n$, where $n \in \omega \setminus \{\zero\}$. 

\medskip
\paragraph{Induction step} Consider a formula $\varphi(x)$ with one free variable $x$, having complexity $n$.

\paragraph{Case I:} Let $\varphi(x) := \psi(x) \wedge \gamma(x)$. If $\lb \varphi(u) \rb_\pa = \one$, then  \[\lb \psi(u) \rb_\pa= \one=\lb \gamma(u) \rb_\pa. \]  By the induction hypothesis, \[\lb \psi(v) \rb_\pa=\one=\lb \gamma(v) \rb_\pa.\]  Hence, $\lb \varphi(v) \rb_\pa = \one$. 

Now, if $\zero < \lb \varphi(u) \rb_\pa < \one$,  then we have:
\[\zero \neq \lb \psi(u) \rb_\pa \mbox{ and } \zero \neq \lb \gamma(u) \rb_\pa \] and \[ \lb \psi(u) \rb_\pa < \one \mbox{ or } \lb \gamma(u) \rb_\pa < \one.\] By the induction hypothesis; \[\zero \neq \lb \psi(v) \rb_\pa \mbox{ and } \zero \neq \lb \gamma(v) \rb_\pa \] and \[\lb \psi(v) \rb_\pa < \one \mbox{ or }  \lb \gamma(v) \rb_\pa < \one.\] Therefore, $\zero < \lb \varphi(v) \rb_\pa < \one$. Similarly, \textbf{Case II}, \textbf{Case III} and \textbf{Case IV} can  be proved.

\paragraph{Case II:} Let $\varphi(x) := \psi(x) \vee \gamma(x)$.

\paragraph{Case III:} Let $\varphi(x) := \psi(x) \to \gamma(x)$.

\paragraph{Case IV:} Let $\varphi(x) := \lnot ~ \psi(x)$.

\paragraph{Case V:} Let $\varphi(x) := \exists y \, \psi(y, x)$. Suppose $\lb \varphi(u) \rb_\pa = \one$. Follows by  Definition \ref{Definition: MTV algebra}(iii) and the induction hypothesis.  Suppose $\zero < \lb \varphi(u) \rb_\pa < \one$. Then, there exists $p \in \VA{\A}$ such that $\zero < \lb \psi(p, u) \rb_\pa < \one$ and there does not exist any $q \in \VA{\A}$ such that $\lb \psi(q, u) \rb_\pa = \one$. The induction hypothesis ensures that $\zero < \lb \psi(p, v) \rb_\pa < \one$ and $\lb \psi(q, v) \rb_\pa \neq \one$, for all $q \in \VA{\A}$. Finally, by  Definition \ref{Definition: MTV algebra}(iii) it follows that $\zero < \lb \varphi(v) \rb_\pa < \one$.

\paragraph{Case VI:} Let $\varphi(x) := \forall y \, \psi(y, x)$. A dual argument of Case V shows that this case holds due to Definition \ref{Definition: MTV algebra}(iv).
\end{proof}

\begin{coro}\label{corrollary: validity of LL in W+}
Let $(\mathbb{A},D)$ be an ultra-designated $\mathsf{Cobounded}$-algebra. For any $u, v \in \VA{\A}$ and any formula $\varphi(x)$ in $\mathcal{L}_\A$ having one free variable $x$, if $\lb u = v \rb_\pa \wedge \lb \varphi(u) \rb_\pa \in D_{\A}$ then $\lb \varphi(v) \rb_\pa \in D_{\A}$.
\end{coro}

From Corollary \ref{corrollary: validity of LL in W+}, the following theorem can be derived.

\begin{teo}\label{theorem: validity of LL}
Let $(\mathbb{A},D)$ be an ultra-designated $\mathsf{Cobounded}$-algebra. Then,  Leibniz's law of indiscernibility of identicals holds in $\VA{\A,~\lb \cdot \rb_\pa}$, i.e., $\VA{\A,~\lb \cdot \rb_\pa} \models_D \llaw$, for all formulas $\varphi$ in the language $\mathcal{L}_\A$.
\end{teo}

\subsection{Validity of the axioms of $\Z$} In this section, we shall prove $\VA{\A, ~ \lb \cdot \rb_\pa} \models_D \Z$ , for any ultra-designated $\mathsf{Cobounded}$-algebra $(\A,D)$.

\begin{lemma}
For any designated $\mathsf{Cobounded}$-algebra $(\A,D)$ and a subset $\{a_i :i \in I \} \cup \{b\}$ of $\A$ (where $I$ is an index set),
\begin{equation}
    (\bigvee_{i \in I}  a_i )  \Rightarrow b = \bigwedge_{i \in I}  (a_i \Rightarrow b).
    \tag{$\dagger$}
\end{equation}
\end{lemma}
\begin{proof}
It is known that, for any two elements $p$ and $q$ of $\A$, $p \Rightarrow q$ is either $\one$ or $\zero$.
\begin{align*}
    (\bigvee_{i \in I}  a_i )  \Rightarrow b = \zero & \mbox{ iff } \bigvee_{i \in I}  a_i \neq 0 \mbox{ and } b = \zero\\
   & \mbox{ iff there exists } j \in I \mbox{ so that } 
    a_j \neq 0 \mbox{ and } b = \zero\\
   & \mbox{ iff } \bigwedge_{i \in I}  (a_i \Rightarrow b) = \zero.
\end{align*}
\end{proof}

\begin{lemma}\label{lemma: Bounded quantification for Wplus valued models} 
For an ultra-designated $\mathsf{Cobounded}$-algebra $(\A,D)$, an element $u \in \VA{\A}$, and a formula $\varphi(x)$ in $\mathcal{L}_\A$, having one free variable $x$,
\begin{equation}
    \lb \forall x \big{(} x \in u  \rightarrow \varphi(x) \big{)} \rb_\pa = \bigwedge_{x \in \dom(u)} \big{(}u(x) \Rightarrow \lb \varphi(x) \rb_\pa\big{)}.
    \tag{$\bq$}
\end{equation} 
\end{lemma}

\begin{proof} By the definition of the assignment function $\lb.\rb_\pa$,
\begin{align*}
    \lb \forall x \big{(} x \in u & \to \varphi(x)\big{)} \rb_\pa \\
    & = \bigwedge_{y \in \VA{\A}} \llbracket \big{(} y \in u  \rightarrow  \varphi(y)   \big{)} \rb_\pa\\
    & = \bigwedge_{y \in \mathbf{V}^{\mathbb{(A)}}} \big{(} \bigvee_{x \in \dom(u)} ( u(x) \wedge \lb y = x  \rb_\pa) \Rightarrow \lb \varphi(y) \rb_\pa \big{)}\\
    & = \bigwedge_{y \in \mathbf{V}^{\mathbb{(A)}}} \bigwedge_{x \in \dom(u)}  \big{(} (u(x)\wedge \lb x=y \rb_\pa) \Rightarrow \lb \varphi(y) \rb_\pa \big{)}, \mbox{ by } (\dagger)\\
    & = \bigwedge_{y \in \mathbf{V}^{\mathbb{(A)}}} \bigwedge_{x \in \dom(u)}  \big{(}u(x)\Rightarrow (\lb x=y \rb_\pa \Rightarrow \llbracket \varphi(y) \rb_\pa) \big{)},\mbox{ since }\lb x = y \rb_\pa \in \{\one, \zero\}\\
    & = \bigwedge_{y \in \mathbf{V}^{\mathbb{(A)}}} \bigwedge_{x \in \dom(u)}  \big{(}u(x)\Rightarrow (\lb x=y \rb_\pa \Rightarrow \llbracket \varphi(x) \rb_\pa) \big{)}, \mbox{ by Corollary \ref{corrollary: validity of LL in W+}}\\
    & = \bigwedge_{y \in \mathbf{V}^{\mathbb{(A)}}} \bigwedge_{x \in \dom(u)}  \big{(}(u(x)\wedge \llbracket x=y \rb_\pa) \Rightarrow \llbracket \varphi(x) \rb_\pa \big{)}.
\end{align*}

\noindent   Moreover, it can be proved that, if $a \leq b$ then $b \Rightarrow c \leq a \Rightarrow c$, for any three elements $a, b,c$ in $\A$. This concludes that,
\begin{align*}
    \bigwedge_{x \in \dom(u)} \big{(}u(x) \Rightarrow \lb \varphi(x) \rb_\pa \big{)} 
    & = \bigwedge_{y \in  \mathbf{V}^{\mathbb{(A)}}}  \bigwedge_{x \in \dom(u)} \big{(}u(x)  \Rightarrow \lb \varphi(x) \rb_\pa \big{)} \\
    & \leq \bigwedge_{y \in  \mathbf{V}^{\mathbb{(A)}}}  \bigwedge_{x \in \dom(u)} \big{(}(u(x) \wedge \lb x=y \rb_\pa) \Rightarrow \lb \varphi(x) \rb_\pa  \big{)}.
\end{align*}

\noindent On the other hand, for any $x \in \dom(u)$,
\begin{align*}
    \bigwedge_{y \in  \mathbf{V}^{\mathbb{(A)}}}  \big{(}(u(x) \wedge \lb x=y \rb_\pa) \Rightarrow \lb\varphi(x) \rb_\pa \big{)} & \leq (u(x) \wedge \lb x=x \rb_\pa ) \Rightarrow \lb \varphi(x) \rb_\pa\\
    & = u(x) \Rightarrow \varphi(x), \mbox{ using Lemma \ref{PropertiesTvaluedmodels}}(i),
\end{align*}
which implies,
\[\bigwedge_{y \in \mathbf{V}^{\mathbb{(A)}}} \bigwedge_{x \in \dom(u)}  \big{(}(u(x)\wedge \llbracket x=y \rb_\pa) \Rightarrow \llbracket \varphi(x) \rb_\pa \big{)} \leq \bigwedge_{x \in \dom(u)}\big{(} u(x) \Rightarrow \varphi(x) \big{)}.\]

\noindent Hence, 
\[\bigwedge_{y \in \mathbf{V}^{\mathbb{(A)}}} \bigwedge_{x \in \dom(u)}  \big{(}(u(x)\wedge \llbracket x=y \rb_\pa) \Rightarrow \llbracket \varphi(x) \rb_\pa \big{)} 
= \bigwedge_{x \in \dom(u)}\big{(} u(x) \Rightarrow \varphi(x) \big{)},\]
and as a conclusion we have,
\[\lb \forall x \big{(} x \in u \to \varphi(x)\big{)} \rb_\pa = \bigwedge\limits_{x \in \dom(u)}\big{(} u(x) \Rightarrow \varphi(x) \big{)}.\]
\end{proof}

We shall now prove that any ultra-designated $\mathsf{Cobounded}$-algebra-valued model validate all the axioms of $\overline{\ZF}$. 

\begin{teo}\label{PAZF} 
Let $(\A,D)$ be an ultra-designated $\mathsf{Cobounded}$-algebra. Then, we have $\VA{\A,~ \lb \cdot \rb_\pa} \models_{D} \overline{\ZF}$. 
\end{teo}

\begin{proof} $\overline{\mathsf{Extensionality}}$: Consider any $u,v \in \VA{\A}$ such that 
\[\lb\forall z \big{(}(z \in x \leftrightarrow z \in y) \wedge (\neg(z \in x) \leftrightarrow \neg(z \in y)) \rb_\pa\in D.\]
Then, we have to prove that $\lb u=v\rb_\pa \in D$ as well. By our assumption, for any $w \in \VA{\A}$ we have $\lb w \in u\leftrightarrow w \in  v\rb_\pa = \one = \lb \neg(w \in u) \leftrightarrow \neg (w \in  v) \rb_\pa$.
This can only be the case if for any $w \in \VA{\A}$,

\begin{enumerate}[label=(\arabic*)]
    \item $\lb w \in u \rb_\pa \in D \mbox{ iff } \lb w \in v\rb_\pa \in D $ and,
    \item $\lb w \in u \rb_\pa^{*} \in D \mbox{ iff } \lb w \in v\rb_\pa^{*}  \in D $.
\end{enumerate}

\noindent From (1) and (2) we have for any $w \in \VA{\A}$,

\begin{enumerate}
    \item[(i)] $\lb w \in u \rb_\pa = \one \mbox{ iff } \lb w \in v\rb_\pa = \one$,
    \item[(ii)] $\zero<\lb w \in u \rb_\pa < \one \mbox{ iff } \zero <\lb w \in v\rb_\pa < \one$,
    \item[(iii)] $\lb w \in u \rb_\pa = \zero \mbox{ iff } \lb w \in v\rb_\pa = \zero$.
\end{enumerate}
\noindent Let $x_1 \in \dom(u)$ be such that $u(x_1)= \one$, so $\lb x_1 \in u \rb_\pa= \one$. By (i), there exists a $y_1 \in \dom(v)$ such that $v(y_1)= \one$ and $\lb x_1 = y_1 \rb_\pa = \one$.  Let $x_2 \in \dom(u)$ be such that $u(x_2) \in D \setminus \{\one\}$, so $\zero<\lb x_2 \in u \rb_\pa < \one$. By (ii), there exists a $y_2 \in \dom(v)$ such that $\zero < v(y_2) < \one$ and $\lb x_2 = y_2 \rb_\pa = \one$. 
Now, applying Theorem \ref{theorem: equality corresponding to paraconsistent assignment function} we get $\lb u = v \rb_\pa = \one$. Therefore, $\lb \overline{\mathsf{Extensionality}} \rb_\pa = \one$. 

\bigskip

$\mathsf{Pairing}:$ We show that for any two $x,y  \in \VA{\mathbb{A}} $ there exists a $z  \in \VA{\mathbb{A}}$ such that 
\[\lb \forall w \big{(} w \in z \rightarrow (w=x \vee w=y)\big{)}   \rb_\pa \wedge \lb \forall w \big{(}(w = x \vee w=y) \rightarrow w \in z \big{)} \rb_\pa \in D.\]

\noindent We begin by showing that,
\begin{align*}
    \lb \forall w \big{(} & w \in z \rightarrow (w=x \vee w=y)\big{)}   \rb_\pa \in D.
\end{align*}
Fix two arbitrary $x,y  \in \VA{\A}$. Let $z$ be such that $\dom(z)=\{x,y\}$ and $\ran(z)=\{\one\}$. By using $\bq$ on the first conjunct we have:
\[ \lb \forall w \big{(} w \in z \rightarrow (w=x \vee w=y)\big{)}    \rb_\pa = \bigwedge_{w \in \dom(z)} \big{(} z(w) \Rightarrow \lb (w=x \vee w=y) \rb_\pa \big{)}.\]

\noindent Now take any $w_0 \in \dom(z)$ such that $z(w_0) \in D$, then due to the construction of $z$ we have  $\lb w_0 = x \vee w_0 = y \rb_\pa = \one$.  Therefore,  
\[\lb \forall w \big{(} w \in z \rightarrow (w=x \vee w=y)\big{)}   \rb_\pa \in D.\]

Now we will show that this is also the case for the second conjunct of $ \mathsf{Pairing}$, i.e., we show,
\[\lb \forall w \big{(}(w = x \vee w=y) \rightarrow w \in z \big{)} \rb_\pa  = \bigwedge_{w \in \VA{\mathbb{A}}} \big{(} \lb (w = x \vee w=y) \rb_\pa \Rightarrow \lb w \in z  \rb_\pa \big{)} \in D.\]
 Take any $w_0 \in \VA{\mathbb{A}}$ and suppose that $\lb w_0 = x \vee w_0 = y \rb_\pa  \in D$. Then, by the construction of $z$,  we get immediately that $\lb w_0 \in z  \rb_\pa \in D$. Therefore, 
\[\lb \forall w \big{(}(w = x \vee w=y) \rightarrow w \in z \big{)} \rb_\pa \in D.\]
Combining the previous results we conclude that that $\lb \mathsf{Pairing} \rb_\pa \in D$.

\bigskip

$\mathsf{Infinity}:$ 
It is enough to show that, 
\[\lb \forall z \neg( z \in \varnothing) \rb_\pa \wedge \lb \varnothing \in \check{\omega} \rb_\pa \wedge \lb  \forall w \big{(} w \in \check{\omega} \to \exists u (u \in \check{\omega} \wedge w \in u ) \big{)} \rb_\pa \in D, \]
where $\varnothing$ is the empty function in $\VA{\A}$ and $\omega$ is the collection of all natural numbers in $\V$.
The image of the first two conjuncts of $ \mathsf{Infinity}$ is clearly in $D$ under the assignment function $\lb \cdot \rb_\pa$. We show that,
\begin{align*}
    \lb \forall &w \big{(} w \in \check{\omega} \rightarrow \exists u(u \in \check{\omega} \wedge w \in u)\big{)}   \rb_\pa \in D.
\end{align*}

\noindent By applying $\bq$ we have: 

$$\lb \forall w \big{(} w \in \check{\omega} \to \exists u (u \in \check{\omega} \wedge w \in u ) \big{)} \rb_\pa =\bigwedge_{w \in \dom(\check{\omega})} \big{(}    \check{\omega}(w) \Rightarrow \lb    \exists u ( u \in \check{\omega} \wedge w \in u ) \rb_\pa \big{)}. $$

\noindent Now take any $\check{w_0} \in  \dom(\check{\omega})$. By the definition of  $\check{\omega}$, we have $\check{\omega}(\check{w_0})= \one$. Therefore, $\check{w_0} \in \omega$ holds in $\V$.  Now due to $\mathsf{Infinity}$ in $\V$ we know that there exists a $u_0 \in \V$ (the successor of $w_0$)  such that $ u_0 \in \omega$  and $ w_0 \in u_0 $ holds in $\V$. Thus $\check{u}_0 \in \VA{\mathbb{A}}$. 
We can check readily that $\lb \check{u}_0 \in \check{\omega} \rb_\pa = \one $ and $\lb w_0 \in u_0 \rb_\pa =  \one$. Therefore, $\lb \exists u(u \in \check{\omega} \wedge \check{w_0} \in u)  \rb_\pa  \in D$. Furthermore, since the choice of  $\check{w_0}$ was arbitrary we have 
\[\lb \forall w \big{(} w \in \check{\omega} \to \exists u (u \in \check{\omega} \wedge w \in u ) \big{)} \rb_\pa \in D.  \]

\bigskip

$\mathsf{Union}:$ We have to prove, for any $u \in \VA{\A}$, there exists an element $v \in \VA{\A}$ such that $\lb \forall x \big{(}x \in v \leftrightarrow \exists y (y \in u  \wedge x \in y ) \big{)} \rb_\pa \in D$. Indeed, it is sufficient to show that,
\[\lb \forall x \big{(}x \in v \rightarrow \exists y (y \in u  \wedge x \in y ) \big{)} \rb_\pa \wedge \lb \forall x \big{(} \exists y (y \in u  \wedge x \in y ) \to x \in v \big{)} \rb_\pa \in D.\]
Take any $u \in \VA{\A}$ and define $v \in \VA{\A}$, as follows: 
$$\dom(v)= \bigcup \{\dom(y)\mid y \in \dom(u)\} \mbox{ and } v(x) = \lb \exists y (y \in u  \wedge x \in y ) \rb_\pa, \mbox{ for } x \in \dom(v).$$
We first prove  that,
$$\lb \forall x \big{(}x \in v \rightarrow \exists y (y \in u  \wedge x \in y ) \big{)} \rb_\pa \in D,$$

\noindent By using $\bq$, we have

\begin{align*}
    & \lb \forall x \big{(}x \in v \rightarrow \exists y (y \in u  \wedge x \in y ) \big{)}\rb_\pa \\
    & =  \bigwedge_{x \in \dom(v)} \big{(} v(x) \Rightarrow \lb  \exists y (y \in u  \wedge x \in y )  \rb_\pa \big{)}\\
    & = \bigwedge_{x \in \dom(v)} \big{(} \lb \exists y (y \in u  \wedge x \in y ) \rb_\pa \Rightarrow \lb  \exists y (y \in u  \wedge x \in y )  \rb_\pa \big{)}\\
    & \in D, \mbox{ since $a \Rightarrow a = \one \in D$, for any element $a \in \A$.}
\end{align*}

We shall now show that $\lb \forall x \big{(} \exists y (y \in u  \wedge x \in y ) \to x \in v \big{)} \rb_\pa \in D$. Let $x_0 \in \VA{\A}$ be an element such that 
$\lb \exists y (y \in u  \wedge x_0 \in y ) \rb_\pa \in D$. The proof will be completed if it can be derived that $\lb x_0 \in v \rb_\pa \in D$, as well.
By definition, $\lb \exists y (y \in u  \wedge x_0 \in y ) \rb_\pa \in D$ implies that, there exists $y_0 \in \VA{\A}$ such that $\lb y_0 \in u ~ \wedge ~ x_0 \in y_0 \rb_\pa \in D$. Now, $\lb y_0 \in u \rb_\pa \in D_{\T}$ guarantees an element $y_1 \in \dom(u)$ such that $u(y_1) \in D$ and $\lb y_1 = y_0 \rb_\pa \in D$. 
So, we have, $\lb y_1 = y_0 ~ \wedge ~ x_0 \in y_0 \rb_\pa \in D$ and by using Lemma \ref{PropertiesTvaluedmodels}$(v)$ it can be concluded that $\lb x_0 \in y_1 \rb_\pa \in D$. Hence, there exists $x_1 \in \dom(y_1)$ satisfying $y_1(x_1) \in D$ and $\lb x_0 = x_1 \rb_\pa \in D$. 
Since, by our assumption, $u(y_1) \in D$ and $y_1(x_1) \in D$, applying Lemma \ref{PropertiesTvaluedmodels}$(ii)$ both $\lb y_1 \in u \rb_\pa \in D$ and $\lb x_1 \in y_1 \rb_\pa \in D$ hold. Hence, $\lb y_1 \in u ~ \wedge ~ x_1 \in y_1 \rb_\pa \in D$, which leads to the fact that $\lb \exists y (y \in u ~\wedge~ x_1 \in y ) \rb_\pa \in D$, i.e., $v(x_1) \in D$. 
So, we have derived that, $\lb x_0 = x_1 \rb_\pa, ~ v(x_1) \in D$. Hence, $\lb x_0 \in v \rb_\pa \in D$.

Therefore, we can conclude  $\lb \mathsf{Union} \rb_\pa \in D$.

$\mathsf{Power set}:$ We have to prove that for any $x \in \VA{\mathbb{A}}$ there exists a $y \in \VA{\mathbb{A}}$ such that
\[\big{(}\lb \forall z \big{(}z \in y \rightarrow \forall w (w \in z  \to w \in x)\big{)} \rb_\pa \wedge \lb \forall z \big{(} \forall w(w \in z \rightarrow w \in x)\rightarrow z \in y  \big{)}\rb_\pa \big{)} \in D.\]
Take any $x \in \VA{\mathbb{A}}$ and define $y$ such that
\[\dom(y)= \mathbb{A}^{\dom(x)} \mbox{ and for any }z \in \dom(y), y(z)= \lb \forall w (w \in z \to w \in x) \rb_\pa.\]
The first conjunct can be proved by an immediate application of the property $\bq$.

We now show that,
\[\bigwedge_{z  \in \VA{\mathbb{A}}} \lb \big{(} \forall w(w \in z \rightarrow w \in x)\rightarrow z \in y  \big{)}\rb_\pa \in D.\]
Let us fix an arbitrary $z \in \VA{\mathbb{A}}$. Then, 
\begin{align*}
 \lb &\forall w(w \in z \rightarrow w \in x) \rb_\pa \Rightarrow  \lb z \in y  \rb_\pa \\
 = &\bigwedge_{w \in \dom(z)} \big{(}  z(w) \Rightarrow \lb w \in x \rb_\pa \big{)} \Rightarrow \bigvee_{q \in \dom(y)} \big{(} y(q) \wedge  \lb z = q \rb_\pa \big{)}, \mbox{ using the property $\bq$}\\
= & \bigwedge_{w \in \dom(z)} \big{(}  z(w) \Rightarrow \lb w \in x \rb_\pa \big{)} \Rightarrow  \bigvee_{q \in \dom(y)} \big{(} \bigwedge_{p \in \dom(q)} \big{(}  q(p) \Rightarrow \lb p \in x \rb_\pa \big{)} \wedge \\
 &\qquad \qquad \bigwedge_{w \in \dom(z)}  \big{(}z(w) \Rightarrow \lb w \in q \rb_\pa \big{)} \wedge \bigwedge_{w \in \dom(z)} \big{(}   \lb w \in q \rb_\pa^{*} \Rightarrow z(w)^{*}  \big{)} \wedge\\ 
 &\qquad \qquad \bigwedge_{p \in \dom(q)} \big{(} q(p) \Rightarrow \lb p \in z \rb_\pa   \big{)} \wedge \bigwedge_{p \in \dom(q)} \big{(}   \lb p \in z \rb_\pa^{*} \Rightarrow  q(p)^{*}  \big{)} \big{)}.\\
\end{align*}
Let us assume that, 
\begin{align*}
&\bigwedge_{w \in \dom(z)} \big{(}  z(w) \Rightarrow \lb w \in x \rb_\pa \big{)} \neq \zero. \tag{$\dagger$}
\end{align*}
Then it is enough to show that there exists an element $q'\in \dom(y)$ so that,
\begin{enumerate}
    \item[(i)] $\bigwedge\limits_{p \in \dom(q')} \big{(} q'(p) \Rightarrow \lb p \in x \rb_\pa \big{)} \neq \zero$,
    \item[(ii)] $\bigwedge\limits_{w \in \dom(z)}  \big{(}z(w) \Rightarrow \lb w \in q' \rb_\pa \big{)} \neq \zero$,
    \item[(iii)] $\bigwedge\limits_{w \in \dom(z)}  \big{(}  \lb w \in q' \rb_\pa^{*} \Rightarrow z(w)^{*} \big{)} \neq \zero$,
    \item[(iv)] $\bigwedge\limits_{p \in \dom(q')} \big{(} q'(p) \Rightarrow \lb p \in z \rb_\pa   \big{)} \big{)} \neq \zero$, and
    \item[(v)] $\bigwedge\limits_{p \in \dom(q')} \big{(}   \lb p \in z \rb_\pa^{*}  \Rightarrow   q'(p)^{*}  \big{)} \neq \zero$.
\end{enumerate}
 Fix $q'\in \dom(y)$ such that $q'(p)=  \lb p \in z \rb_\pa $, for all $p \in \dom(q')$. We now prove that all the properties (i) - (v) hold for the element $q'\in \dom(y)$.

\begin{enumerate}
    \item[(i)] Notice that due to $(\dagger)$ we have $\lb \forall w(w \in z \rightarrow w \in x) \rb_\pa \neq \zero$ \begin{align*} &\mbox{i.e.,} \bigwedge_{w \in \VA{\mathbb{A}}} \big{(} \lb w \in z \rb \Rightarrow \lb w \in x \rb_\pa \big{)} \neq \zero\\
      &\mbox{i.e.,} \bigwedge_{p \in \dom(q')} \big{(} \lb p \in z \rb_\pa  \Rightarrow \lb p \in x \rb_\pa \big{)} \neq \zero\\ 
       &\mbox{i.e.,} \bigwedge_{p \in \dom(q')} \big{(}  q'(p)   \Rightarrow \lb p \in x \rb_\pa \big{)} \neq \zero 
    \end{align*}

\item[(ii)] In order to show that $\bigwedge\limits_{w \in \dom(z)}  \big{(} z(w) \Rightarrow  \lb w \in q' \rb_\pa \big{)} \in D$, if $z(w) =\zero$, then we are done. So, let $z(w) \neq \zero$. Then, due to $(\dagger)$ we can conclude that $\lb w \in x \rb_\pa \neq \zero$, i.e., 
\[\bigvee_{p \in \dom(x)}(x(p) \wedge \lb p = w \rb_\pa) \neq \zero.\] Thus there exists a $p' \in  \dom(x)$ such that $x(p')\neq \zero \mbox{ and } \lb p'=w \rb_\pa = \one$. Hence, $\lb p' \in z  \rb\geq (z(w)\wedge \lb p'=w \rb_\pa ) \neq \zero$ by our assumptions and by Definition \ref{Definition: MTV algebra}(iv). Hence, we get the following,
\begin{align*}
    \lb w \in q' \rb_\pa & = \bigvee_{p \in \dom(q')}(q'(p) \wedge \lb p = w \rb_\pa)\\
    & \geq (q'(p') \wedge \lb p' = w \rb_\pa), \mbox{ since $p' \in \dom(q')$ }\\
    & = (  \lb p' \in z \rb_\pa \wedge \lb p' = w \rb_\pa)\\
    &  \neq \zero.
\end{align*}
Hence, $\bigwedge_{w \in \dom(z)}\limits \big{(}z(w) \Rightarrow \lb w \in q' \rb_\pa \big{)}  \neq \zero$, by Definition \ref{Definition: MTV algebra}(iv). 
\item[(iii)] We show that $\bigwedge\limits_{w \in \dom(z)}  \big{(}  \lb w \in q' \rb_\pa^{*} \Rightarrow z(w)^{*} \big{)} \in D$. If $z(w)\neq \one$ then we are done. So let $z(w') = \one$ for some  $w'\in \dom(z)$. Then we have to prove that $\lb w' \in q' \rb_\pa = \one$. Now, 
\begin{align*}
    \lb w' \in q' \rb_\pa & = \bigvee_{p \in \dom(q')}(q'(p) \wedge \lb p = w' \rb_\pa)\\
    & =  \bigvee_{p \in \dom(q')}(  \lb p \in z \rb_\pa \wedge \lb p = w' \rb_\pa).\\
\end{align*} 
But then, due to $(\dagger)$, we know that there exists a $p'  \in \dom(x)$ such that $x(p')\neq \zero$ and $\lb p' =w'  \rb_\pa = \one$.  Thus, $\lb p\in z\rb \geq (z(w' ) \wedge \lb w' = p' \rb_\pa) = \one$. Therefore, we can conclude $ \lb w' \in q' \rb_\pa  = \one$.
\item[(iv)]  We show that $\bigwedge\limits_{p \in \dom(q')} \big{(} q'(p) \Rightarrow \lb p \in z \rb_\pa   \big{)} \in D$. Notice,
\begin{align*}&\bigwedge_{p \in \dom(q')} \big{(} q'(p) \Rightarrow \lb p \in z \rb_\pa   \big{)}\\
&=\bigwedge_{p \in \dom(q')} \big{(} \lb p \in z \rb_\pa  \Rightarrow \lb p \in z \rb_\pa   \big{)}\\
&\in D, \mbox{ since for any } a \in \mathbf{A}, (a \Rightarrow a)= \one.\\
\end{align*}
    
\item[(v)] Same as (iv).



\end{enumerate} Finally, we can conclude that $\lb \mathsf{Powerset} \rb_\pa \in D$.  

\bigskip 

$\mathsf{Separation}:$ Let $\varphi(x)$ be any formula in $\mathcal{L}_\A$, where $x$ is the only free variable. We want to show that for any  $x \in \VA{\mathbb{A}}$ there exists a $y \in \VA{\mathbb{A}}$ such that 
\[\lb \forall z \big{(}z \in y  \leftrightarrow  (z \in x \wedge \varphi(z))\big{)}\rb_\pa \in D.\]
It is sufficient to show that: 
\[\lb \forall z \big{(}z \in y  \rightarrow  (z \in x \wedge \varphi(z))\big{)} \rb_\pa \wedge \lb \forall z \big{(}( z \in x \wedge \varphi(z)) \rightarrow  z \in y\big{)} \rb_\pa \in D. \]
For any $x \in \VA{\mathbb{A}}$ define $y \in \VA{\mathbb{A}}$ as follows: 
$$\dom(y)= \dom(x) \mbox{ and for any } z \in  \dom(y) \mbox{ let } y(z)= x(z) \wedge \lb \varphi(z) \rb_\pa.$$
We first prove that,
$$\lb \forall z \big{(}z \in y  \rightarrow  (z \in x \wedge \varphi(z))\big{)} \rb_\pa \in D.$$

By using $\bq$ , we have

\begin{align*}
    & \lb \forall z \big{(}z \in y \rightarrow  (z \in x  \wedge \varphi(z) ) \big{)}\rb_\pa \\
    & =  \bigwedge_{z \in \dom(y)} \big{(} y(z) \Rightarrow \lb z \in x  \wedge \varphi(z) \rb_\pa \big{)}\\
    & = \bigwedge_{z \in \dom(y)} \big{(}  (x(z) \wedge \lb \varphi(z) \rb_\pa)  \Rightarrow (\lb z \in x \rb_\pa  \wedge \lb \varphi(z)  \rb_\pa ) \big{)}\\
    & \in D, \mbox{ by Lemma \ref{PropertiesTvaluedmodels}(ii)}.
\end{align*}

Now we show that the second conjunct of $\mathsf{Separation}$ holds as well. Since, for any $a,b,c$ in $\A$, 
$(a \wedge b) \Rightarrow c = a \Rightarrow (b \Rightarrow c)$, we have 
\begin{align*}
\bigwedge_{z \in \VA{\mathbb{A}}} \lb ( z \in x \wedge \varphi(z)) \rightarrow z \in y  \rb_\pa & = \bigwedge_{z \in \VA{\mathbb{A}}} \big{(}\lb z \in x \rb_\pa \Rightarrow (\lb \varphi(z) \rb_\pa \Rightarrow \lb z \in y \rb_\pa) \big{)}\\
& = \lb \forall z \big{(} z \in x \rightarrow (\varphi(z) \rightarrow z \in y)\big{)}\rb_\pa \\
& = \bigwedge_{z \in \dom(x)}\big{(} x(z) \Rightarrow  \lb(\varphi(z) \Rightarrow z \in y ) \rb_\pa \big{)}.
\end{align*}
For any $z_0 \in \dom(x)$, suppose $x(z_0), \lb  \varphi(z_0)  \rb_\pa  \in D$. Then by construction of $y$ we have $y(z_0) \in D$ and by lemma \ref{PropertiesTvaluedmodels}(ii) $\lb z_0 \in y \rb_\pa \in D$. Therefore, we can conclude that for any $z_0 \in \dom(x)$,  $\big{(}x(z_0) \Rightarrow  \lb(\varphi(z_0) \Rightarrow z_0 \in y ) \rb_\pa) \big{)} \in D$. Hence, 
\[\lb \forall z \big{(} (z \in x \wedge \varphi(z)) \rightarrow  (z \in y) \big{)} \rb_\pa \in D.\]

 Since the images of both the conjuncts of $\mathsf{Separation}$ are in $D$ under the assignment function $\lb \cdot \rb_\pa$, we  conclude that $\lb \mathsf{Separation} \rb \in D$.

\bigskip

$\mathsf{Collection}:$ Let $\varphi(x,y)$ be any formula in the language of set theory with two free variables.  We want to proof that for every $u \in \VA{\mathbb{A}}$ we have 
\[\lb  \forall x\big{(}x \in u \to \exists y \varphi (x, y)\big{)} \rightarrow \exists v \forall x\big{(}x \in u \to \exists y(y \in v \wedge \varphi (x, y))\big{)} \rb_\pa \in D.\]
Take any $u \in \VA{\mathbb{A}}$ and assume the antecedent holds, so  $\lb  \forall x\big{(}x \in u \to \exists y \varphi (x, y)\big{)} \rb_\pa \in D$. In particular, by using $\bq$, we have;
\[\tag{1} \lb\forall x\big{(}x \in u \rightarrow \exists y \varphi (x, y)\big{)} \rb_\pa    =  \bigwedge_{x \in \dom(u)} \big{(}u(x) \Rightarrow \bigvee_{y \in \VA{\mathbb{A}} }  \lb \varphi(x,y) \rb_\pa \big{)} \in D.\]
Now we will show that there exists a $v \in \VA{\mathbb{A}}$ such that
\[\lb\forall x\big{(}x \in u \rightarrow \exists y(y \in v \wedge \varphi (x, y))\big{)} \rb_\pa \in D.\]

We  know that  $\mathbb{A}$ is  a set, so $\mathbb{A} \in \mathbf{V}$. Thus, we may apply $\mathsf{Collection}$ in $\V$ so that for any $x \in \dom(u)$ we obtain an ordinal $\alpha_x$ such that  
$$ \bigvee_{y \in \VA{\mathbb{A}}} \lb \varphi(x,y) \rb_\pa = \bigvee_{y \in \VA{\mathbb{A}}_{\alpha_x}}\lb \varphi(x,y) \rb_\pa.$$
So we have 
\[ \tag{2} \bigwedge_{x \in \dom(u)} \big{(} u(x) \Rightarrow \bigvee_{y \in \VA{\mathbb{A}}_{\alpha_x}}  \lb \varphi(x,y) \rb_\pa \big{)} \in D.\] 
We apply the union axiom in $\V$ to define $\alpha = \bigcup\{ \alpha_x : x \in \dom(u)\}$.  We define the element $v \in \VA{{\mathbb{A}}}$ as $\dom(v)= \VA{\mathbb{A}}_{\alpha}$ and  for every $y \in \dom(v), v(y) = \one$. We move on to show, $$\bigwedge_{x \in \dom(u)} \big{(} u(x)  \Rightarrow \lb \exists y\big{(}y \in v \wedge \varphi (x, y)\big{)} \rb_\pa \big{)} \in D.$$
Take any $x_0 \in \dom(u)$ such that $u(x_0) \in D$. By $(2)$ we have $y_0 \in \VA{\mathbb{A}}_{\alpha_{x_0}}$ such that $\lb \varphi(x_0,y_0) \rb_\pa \in D$. By our construction $y_0 \in \dom(v)$ and $v(y_0) = \one$. Then, it follows by Lemma $\ref{PropertiesTvaluedmodels}$(ii) that $\lb y_0 \in v \rb \in D$.  Therefore, $\lb y_0 \in v \wedge \varphi(x_0,y_0) \rb_\pa \in D$ and thus,
\[\lb \exists y \big{(}y \in v \wedge \varphi(x_0,y) \big{)} \rb_\pa \in D.\]
Since the choice of $x_0$ is arbitrary we have, 
\[  \lb\forall x\big{(}x \in u \to \exists y(y \in v \wedge \varphi (x, y))\big{)} \rb_\pa \in D. \]

Therefore, we conclude that $\lb \mathsf{Collection}\rb \in D$.

\bigskip

$\mathsf{Foundation}:$ We want to show that 
\[\lb \forall x \big{(}\forall y(y \in x \to \varphi(y)) \rightarrow \varphi(x)  \big{)} \rightarrow \forall x \varphi(x) \rb_\pa \in D.\]

Take any $x \in \VA{\mathbb{A}}$ and consider the following two cases:

\noindent (i) Let $\lb \varphi (x) \rb_\pa \in D$ for every $x \in \VA{\mathbb{A}}$, which implies $\lb \forall x \varphi(x) \rb_\pa \in D$.
Therefore,  we get  
\[\lb \forall x \big{(}(\forall y(y \in x \to \varphi(y))) \rightarrow \varphi(x)\big{)} \rightarrow \forall x \varphi(x) \rb_\pa \in D.\]
(ii) Let $\lb \varphi (x) \rb_\pa \notin D$ for some $x \in \VA{\mathbb{A}}$. 
Then, take a minimal $u \in \VA{\mathbb{A}}$ such that $\lb\varphi(u) \rb_\pa \notin D$ and for any $v \in \dom(u)$, $\lb\varphi(v) \rb_\pa \in D$. For this $u$, we claim that 
\[\lb \forall y(y \in u \to \varphi(y)) \rb_\pa     \in D.\]
Using $\bq$ and our assumption, it is immediate that $\lb \forall y(y \in u \to \varphi(y))\rb_\pa \in D$.  From this, we can conclude that

\begin{align*}
&\lb \forall x  \big{(} (\forall y(y \in x \to \varphi(y))  \to \varphi (x) \big{)} \rb_\pa\\
&\leq 
\lb  \big{(}\forall y(y \in u \to \varphi(y)) \rb_\pa \Rightarrow \lb \varphi (u)  \rb_\pa\\
&= \zero.
\end{align*}
Hence, we can finally conclude that
\[\lb \forall x \big{(}\forall y(y \in x \to \varphi(y)) \rightarrow \varphi(x)  \big{)} \rightarrow \forall x \varphi(x) \rb_\pa \in D,\]
i.e., $\lb \mathsf{Foundation} \rb_\pa \in D$.

\end{proof}

\section{Logic(s) corresponding to an algebra $\A$ and $\VA{\A}$}

Let $\mathsf{Prop}$ and $\mathsf{Sent}$ be the collection of all propositional formulas and all sentences of the language $\mathcal{L}_\in$, respectively.
\begin{defi}
Any homomorphism $\tau : \mathsf{Prop} \to \mathsf{Sent}$ is said to be a translation function.
\end{defi}

\noindent Then, following  \cite{illoyal}, we define  the logic of the set theory $\VA{\A}$ under a given assignment function $\lb \cdot \rb$, where $\A$ is a complete distributive lattice.
\begin{defi}
Let $\A = \langle \mathbf{A}, \wedge, \vee, \Rightarrow, ^*, \one, \zero \rangle$ be an algebra, where $\langle \mathbf{A}, \wedge, \vee, \one, \zero \rangle$ is a complete distributive lattice. 
Corresponding to a given assignment function $\lb \cdot \rb$ and a designated set $D$, the logic of the set theory $\VA{\A}$, denoted by $\mathbf{L}(\lb\cdot\rb, \VA{\A})$, is the collection of all propositional formulas $\alpha$ such that $\lb \tau(\alpha) \rb \in D$, for all translation functions $\tau$, i.e., 
\[\mathbf{L}(\lb\cdot\rb, \VA{\A}):= \{\alpha \in \mathsf{Prop}: \mbox{ for all translation function } \tau, \lb \tau(\alpha) \rb \in D\}.\]
The logic of the algebra $\A$ with respect to the designated set $D$, denoted by $\mathbf{L}(\A, D)$, is defined as follows:
\[\mathbf{L}(\A, D):= \{\alpha \in \mathsf{Prop}: \mbox{ for all valuation } v, v(\alpha) \in D\}.\]
\end{defi}

For any Boolean algebra $\mathbb{B}$, $\mathbf{L}(\mathbb{B}, \{\one\})$ is nothing but the classical propositional logic. Since, the two-valued designated $\co$-algebra $(\A, D)$ is the two-valued Boolean algebra, $\mathbf{L}(\A, D)$ is the classical propositional logic. Note that in this case $D = \{\one\}$.

\subsection{Logic of a designated $\co$-algebra}

We shall now prove that, for any designated $\co$-algebra $(\A, D)$ having at least two elements in $D$, the logic $\mathbf{L}(\A, D)$ is \emph{paraconsistent}: 
for an algebra $\A$ and a designated set $D$, the logic $\mathbf{L}(\A, D)$ is said to be paraconsistent if there exist propositional formulas $\alpha$ and $\beta$ such that $(\alpha \wedge \neg \alpha) \to \beta \notin \mathbf{L}(\A, D)$.

\begin{teo}\label{theorem: L(A, D) is paraconsistent}
If $(\A, D)$ is a designated $\co$-algebra, where $D$ contains at least two elements, then $\mathbf{L}(\A, D)$ is paraconsistent.
\end{teo}

\begin{proof}
By the given condition $D \setminus \{\one\} \neq \varnothing$. Let $a \in D \setminus \{\one\}$. By the definition of designated $\co$-algebra, $a^* = a$. 
Consider two propositional formulas $\alpha$ and $\beta$, and a valuation function $v$ such that $v(\alpha) = a$ and $v(\beta) = \zero$. 
Then, 
\[v((\alpha \wedge \neg \alpha) \to \beta) = \zero \notin D.\]
Hence, $(\alpha \wedge \neg \alpha) \to \beta \notin \mathbf{L}(\A, D)$, proving that $\mathbf{L}(\A, D)$ is paraconsistent.
\end{proof}

In Section \ref{section: the algebra PS3}, we introduced the algebra $\ps$ having the designated set $\{1, \nicefrac{1}{2}\}$. A paraconsistent logic $\mathbb{L}\ps$ was found which is sound and complete with respect to $\ps$ \cite[Theorem 3.1 \& 3.6]{Tarafder-Chakraborty}. 
We shall now prove that any ultra-designated $\co$-algebra (which is not a Boolean algebra) is also sound and complete with respect to $\mathbb{L}\ps$. 
Let $(\A, D)$ be an ultra-designated $\co$-algebra. We can check that the function $f_\A : \A \to \s$ remains a homomorphism as $f_\A(a^*) = \big{(}f_\A(a)\big{)}^*$, for all $a \in \mathbf{A}$. We shall use this fact to prove the following soundness and completeness theorem.

\begin{teo}
Let $(\A, D)$ be an ultra-designated $\co$-algebra having more than two elements. Then, $\mathbb{L}\ps$ is sound and complete with respect to $(\A, D)$.
\end{teo}

\begin{proof}
As usual, the proof of soundness is straight forward. We shall prove the completeness only. 
For this, we first fix that $\mathsf{Prop}$ denotes the collection of all propositional formulas. 
Let $\varphi \in \mathsf{Prop}$ be such that $\models_\A \varphi$, where $\models$ has its usual meaning. If possible let  $\not\models_{\ps} \varphi$.
Then, there exists a valuation $v_{\ps} : \mathsf{Prop} \to \ps$ such that $v_{\ps}(\varphi) = 0$. 
Let $p_1, p_2, \ldots, p_n$ be the propositional variables present in $\varphi$. 
Let us consider a valuation $v_{\A} : \mathsf{Prop} \to \A$, which satisfies the following: for each $i = 1, \ldots, n$
\begin{align*}
   v_\A(p_i) & = \left \{\begin{array}{ll}
                        \one, & \mbox{if } v_{\ps}(p_i) = 1;  \\
                        a, & \mbox{if } v_{\ps}(p_i) = \nicefrac{1}{2};\\
                        \zero, & \mbox{if } v_{\ps}(p_i) = 0,
                        \end{array}\right.
\end{align*}
where $a \in \mathbf{A} \setminus \{\one, \zero\}$ is arbitrarily fixed; the existence of such an element is ensured by the fact that $\mathbf{A}$ (the domain of $\A$) has more than two elements. 
Then, by definition, $f_\A\big{(}v_\A(p_i)\big{)} = v_{\ps}(p_i)$, for all $i = 1, \ldots, n$. Since $f_\A$ is a homomorphism, $f_\A\big{(}v_\A(\varphi)\big{)} = v_{\ps}(\varphi)$. 
Hence, by our assumption $f_\A\big{(}v_\A(\varphi)\big{)} = 0$. This implies that $v_\A(\varphi) = \zero$, which contradicts the fact that $\models_\A \varphi$. Therefore, we conclude that $\models_{\ps} \varphi$. Since $\mathbb{L}\ps$ is complete with respect to $\ps$, finally we get that $\vdash \varphi$, where $\vdash$ has its usual meaning. This completes the proof.
\end{proof}

\subsection{The logics of Boolean-valued models under $\lb \cdot \rb_\ba$ and $\lb \cdot \rb_\pa$}

Using Theorem \ref{theorem: importance of PA-assignment}, we know that the two interpretations $\lb \cdot \rb_\ba$ and $\lb \cdot \rb_\pa$ agree upon the validity of set theoretic sentences when a Boolean-valued model is concerned. Hence, we obtain the following theorem.

\begin{teo}
For any complete Boolean algebra $\A$ with $\{\one \}$ as designated set, we have that  $\mathbf{L}(\lb\cdot\rb_\ba, \VA{\A}) = \mathbf{L}(\lb\cdot\rb_\pa, \VA{\A})$.
\end{teo}
It is yet not known whether for any complete Boolean algebra $\A$, where the designated set contains only the top element, $\mathbf{L}(\lb\cdot\rb_\ba, \VA{\A})$ is the classical propositional logic. 
The result holds for any atomic Boolean algebra\footnote{A complete Boolean algebra $\A$ is called \emph{atomic} if the set of atoms of $\A$ is non-empty and for any element $a$ of the Boolean algebra, there exists a sebset $S$ of the set of all atoms of $\A$ such that $a = \bigvee S$.}. However, the atomicity is not a necessary condition for this \cite[Corollary 4.5]{illoyal}. Hence, for a complete atomic Boolean algebra $\A$, $\mathbf{L}(\lb\cdot\rb_\pa, \VA{\A})$ is the classical propositional logic.

\subsection{The logics of $\mathsf{Cobounded}$-algebra-valued models under $\lb \cdot \rb_\ba$ and $\lb \cdot \rb_\pa$}

We shall now prove that the logics of $\VA{\A}$, for any designated $\mathsf{Cobounded}$-algebra having at least two elements in the designated set, under both the assignment functions $\lb \cdot \rb_\ba$ and $\lb \cdot \rb_\pa$ are \emph{paraconsistent}: for a complete lattice $\A$, a designated set $D$, and an assignment function $\lb \cdot \rb$, the logic $\mathbf{L}(\lb\cdot\rb, \VA{\A})$ is said to be paraconsistent if there exist $\varphi, \psi \in \mathsf{Sent}$ such that $(\varphi \wedge \neg \varphi) \to \psi \notin \mathbf{L}(\lb\cdot\rb, \VA{\A})$.

\begin{lemma}\label{lemma: (T, BA) is paraconsistent}
There exists a formula $\varphi \in \mathsf{Sent}$ such that both
$\VA{\mathbb{A}, ~ \lb \cdot \rb_\ba} \models_D \varphi \wedge \neg \varphi$ and $\VA{\mathbb{A}, ~ \lb \cdot \rb_\pa} \models_D \varphi \wedge \neg \varphi$ hold, for any designated $\mathsf{Cobounded}$-algebra $(\mathbb{A},D)$, where $D$ contains at least two elements.
\end{lemma}

\begin{proof} 
Consider the sentence: $\varphi := \exists x \exists y (x \in y \wedge x \notin y )$ and an arbitrary designated $\mathsf{Cobounded}$-algebra $(\mathbb{A},D)$, where $D$ contains at least two elements. The proof will be done only for the assignment function $\lb \cdot \rb_\ba$ and the same proof will hold for the $\lb \cdot \rb_\pa$-assignment function as well.

For any two $u, v \in \VA{\A}$, $\lb u \in v \rb_\ba \wedge \lb u \in v \rb_\ba^* \neq \one$. The designated set $D$ contains elements other than $\one$.
For any element $a \in D \setminus \{\one\}$, consider two elements $ u,v \in \mathbf{V}^{\mathbb{(A)}}$ such that $u$ is arbitrary and $v= \{\langle u, a \rangle \}$. 
We readily calculate that $\lb u \in v \rb_\ba = a$, as well as $\lb u \in v \rb^*_\ba = a^* = a$. 
Hence, we can conclude that for any element $a \in D \setminus \{\one\}$, there exist two elements $ u,v \in \mathbf{V}^{\mathbb{(A)}}$ such that $\lb u \in v \rb_\ba \wedge \lb u \in v \rb^*_\ba = a$. 
This implies that $\lb \varphi \rb_\ba = c$, where $c$ is the coatom of $\A$.
Finally we can derive that 
\[\lb \varphi \rb_\ba = c = c^* = \lb \neg \varphi \rb_\ba \in D,\]
which completes the proof.
\end{proof}

\begin{teo}\label{Theorem: models are paraconsistent}
For any designated $\mathsf{Cobounded}$-algebra $(\A,D)$, where $D$ contains at least two elements, both the logics $\mathbf{L}(\lb\cdot\rb_\ba, \VA{\A})$ and $\mathbf{L}(\lb\cdot\rb_\pa, \VA{\A})$ are paraconsistent.
\end{teo}

\begin{proof}
Consider an arbitrary designated $\mathsf{Cobounded}$-algebra $(\A,D)$, where $D$ contains at least two elements and suppose $\varphi$ is the same formula used in the proof of Lemma \ref{lemma: (T, BA) is paraconsistent}. 
Consider the formula $\psi := \neg \forall x (x = x)$.
Clearly, $\lb \psi \rb_\ba = \zero = \lb \psi \rb_\pa$, and we already know that $\lb \varphi \rb_\ba = c = \lb \varphi \rb_\pa$, where $c$ is the coatom of $\A$. Hence, 
\[\lb (\varphi \wedge \neg \varphi) \to \psi \rb_\ba = \zero = \lb (\varphi \wedge \neg \varphi) \to \psi \rb_\pa.\]
So we can conclude that, there exist formulas $\varphi$ and $\psi$ in $\mathcal{L}_\in$ such that the formula 
\[(\varphi \wedge \neg \varphi) \to \psi \notin \mathbf{L}(\lb\cdot\rb_\ba, \VA{\A}), \mathbf{L}(\lb\cdot\rb_\pa, \VA{\A}).\]

\end{proof}

\section{The quotient space $\VA{\A}/\!\sim$}

In order to simplify notation, in this section we drop the subscript from the assignment function $\lb\cdot\rb_\pa$ and we will denote it only by $\lb \cdot \rb$. Consequently, the validity of a sentence $\varphi$ in a designated $\mathsf{Cobounded}$-algebra-valued model $\VA{\A}$ will be denoted by $\VA{\A} \models \varphi$ instead of $\VA{\A, ~\lb \cdot \rb_\pa} \models \varphi$.

\begin{defi}
A class relation $\sim$ in $\VA{\A}$ is defined as follows $u \sim v$ iff $\VA{\A} \models u = v$, for $u, v \in \VA{\A}$.
\end{defi}

\begin{observation} The relation $ ``\sim"$ is an equivalence relation on $\VA{\A}$.
\end{observation}

\begin{proof}
The proof follows from the fact that for any $u, v \in \VA{\A}$, either $\lb u = v \rb_\pa = \one$ or $\lb u = v \rb_\pa = \zero$ and Theorem \ref{theorem: equality corresponding to paraconsistent assignment function}.
\end{proof}

Once we prove that in a designated $\textsf{Cobounded}$-algebra-valued model $\VA{\A}$ the relation $\sim$ is an equivalence relation, then, we can construct the \emph{quotient space} $\VA{\A}/\!\sim$, as in the classical case. 
The elements of $\VA{\A}/\!\sim$ are of the form $[u]= \{v \in \VA{\A}:\VA{\A} \models u=v\}$, where $u \in \VA{\A}$. Moreover, given the validity of $\llaw$ in $\VA{\A}$, we can define a notion of satisfaction and validity of a formula in $\VA{\A}/\!\sim$ as follows:

\begin{defi}\label{Definition: semantics of the quotient space} Let $(\mathbb{A},D)$ be an ultra-designated $\mathsf{Cobounded}$-algebra. Suppose $\varphi(x_1, \ldots, x_n)$ is a formula in $\mathcal{L}_{\in}$ having the list (possibly empty) of free variables $x_1, \ldots, x_n$, 
then for any $u_1, \ldots, u_n \in \VA{\A}$, the sequence $([u_1], \ldots, [u_n])$ of $\VA{\A}/\!\sim$ is said to satisfy the formula $\varphi(x_1, \ldots, x_n)$, denoted by $\VA{\A}/\!\sim~ \models \varphi([u_1], \ldots, [u_n])$, 
if $\VA{\A} \models \varphi(u_1, \ldots, u_n)$ holds. A formula is said to be valid in $\VA{\A}/\!\sim$ if it is satisfied by every sequence of $\VA{\A}/\!\sim$.
\end{defi}

\begin{observation}
Using Theorem \ref{theorem: validity of LL}, it can be checked that Definition \ref{Definition: semantics of the quotient space} does not depend on the representatives.
\end{observation}

 By the definition of validity in $\VA{\A}/\!\sim$, a sentence $\varphi$ of $\mathcal{L}_\in$ is valid in $\VA{\A}/\!\sim$ if and only if it is valid in the algebra-valued model $\VA{\A}$, i.e.,  $\VA{\A}/\!\sim ~\models \varphi$ iff $\VA{\A} \models \varphi$. Then, as a corollary of Theorem \ref{PAZF}, we get the following theorem.

\begin{teo}
Let $(\A,D)$ be an ultra-designated $\mathsf{Cobounded}$-algebra. Then, $\VA{\A}/\!\sim ~ \models \overline{\ZF}$.
\end{teo}

In the case of building quotient models out of Boolean-valued models, first the interpretation of the predicate symbols $\in$ and $=$ is given and only then a definition of the validity of the rest of the formulas is provided. 
Contrary to this order, in this paper, we first define the validity of all formulas in Definition \ref{Definition: semantics of the quotient space} and only later we provide the full interpretation of $\in$ and $=$, which includes also the case of non-validity.

\begin{defi}\label{Definition: Quotient space identity and belongniness} 
We define the following binary relations on $\VA{\A}/\!\sim$. For any $u, v \in \VA{\A}$,
\begin{align*}
        (i) &  \left \{ \begin{array}{ll}
                        R_{=}([u], [v]) & \mbox{iff  }~ \VA{\A} \models u=v, \\
                        R_{\neq}([u], [v]) & \mbox{iff  }~ \VA{\A} \models \neg (u=v),
                      \end{array}\right.\\
     (ii) &  \left \{\begin{array}{ll}
                       R_{\in}([u], [v]) & \mbox{iff  }~ \VA{\A} \models u \in v,  \\
                        R_{\notin}([u], [v]) & \mbox{iff  }~ \VA{\A} \models \neg (u \in v). 
                      \end{array}\right.
\end{align*}

\end{defi}

\begin{observation}
Applying Theorem \ref{theorem: validity of LL}, one can prove that the relations $R_{=}([u], [v])$, $R_{\neq}([u], [v])$, $R_{\in}([u], [v])$, and $R_{\notin}([u], [v])$ are well-defined.
\end{observation}

As expected, the interpretations of the predicate symbols $=$ and $\in$ in the quotient space $\VA{\A}/\!\sim$ are the binary relations $R_=$ and $R_\in$, respectively. 
But, the construction differs from the classical case when non-satisfaction is considered. 
In Theorem \ref{theorem: equality and its negation in the quotient space} and Theorem \ref{theorem: belongingness and its negation in the quotient space}, we discuss this issue in full detail. 

To keep the expressions simple, for any two elements $[u], [v] \in \VA{\A}/\!\sim$ we shall abbreviate $(x = y)([u], [v])$, $\neg(x = y)([u], [v])$, $(x \in y)([u], [v])$, and $\neg(x = y)([u], [v])$ as $[u] = [v]$, $[u] \neq [v]$, $[u] \in [v]$, and $[u] \notin [v]$, respectively. Then, using Definition \ref{Definition: semantics of the quotient space} and Definition \ref{Definition: Quotient space identity and belongniness} we get the following fact.

\begin{fatto}\label{fact: satisfactions of atomic formulas}
For any $[u], [v] \in \VA{\A}/\!\sim$,
\begin{align*}
        (i) &  \left \{ \begin{array}{ll}
                        \VA{\A}/\!\sim~ \models [u] = [v] & \mbox{iff }  R_{=}([u], [v]) \\
                        \VA{\A}/\!\sim~ \models [u] \neq [v] & \mbox{iff }  R_{\neq}([u], [v]);
                      \end{array}\right.\\
     (ii) &  \left \{\begin{array}{ll}
                        \VA{\A}/\!\sim~ \models [u] \in  [v] & \mbox{iff }  R_{\in}([u], [v]);  \\
                         \VA{\A}/\!\sim~ \models [u] \notin [v] & \mbox{iff } R_{\notin}([u], [v]); 
                      \end{array}\right.
\end{align*}
\end{fatto}


\begin{teo}\label{theorem: equality and its negation in the quotient space}
For any ultra-designated $\co$-algebra $(\A, D)$ and a pair of elements $[u]$ and $[v]$ of $\VA{\A}/\!\sim$,
\[\VA{\A}/\!\sim~ \models  [u] \neq [v] \mbox{ iff }~ \VA{\A}/\!\sim~ \not\models [u] = [v].\]
\end{teo}

\begin{proof}
Consider any two elements $[u], [v]$ of $\VA{\A}/\!\sim$. Then,
\begin{align*}
    \VA{\A}/\!\sim~ \not\models [u] = [v] & \mbox{ iff}~([u], [v]) \notin R_=, & \mbox{by Fact} ~\ref{fact: satisfactions of atomic formulas}(i)\\
    &\mbox{ iff}~\VA{\A} \not\models u = v, & \mbox{by Definition}~ \ref{Definition: Quotient space identity and belongniness}(i)\\
    &\mbox{ iff}~ \lb u = v \rb = \zero, & \mbox{since } \lb u=v \rb \in \{\zero, \one\} \\
    &\mbox{ iff} ~\VA{\A} \models \neg(u = v), & \mbox{ since $\zero^* = \one \in D$}\\
    &\mbox{ iff}~\VA{\A}/\!\sim~ \models  [u] \neq [v], & \mbox{ by Definition \ref{Definition: Quotient space identity and belongniness}(i) and Fact \ref{fact: satisfactions of atomic formulas}(i).}
\end{align*}
\end{proof}

As a corollary we obtain that the two (class) relations $R_=$ and $R_{\neq}$ are disjoint. Moreover, it is easy to see that $R_=^c$, the complement of $R_=$, is the class $R_{\neq}$. On the other hand, the two class relations $R_\in$ and $R_{\notin}$ are exhaustive but not necessarily exclusive. This is the content of the following theorem.

\begin{teo}\label{theorem: belongingness and its negation in the quotient space}
Let $(\A, D)$ be an ultra-designated $\co$-algebra. Then,
\begin{enumerate}
    \item[(i)] for any $[u], [v] \in \VA{\A}/\!\sim$, if $\VA{\A}/\!\sim~ \not\models [u] \in  [v]$, then $\VA{\A}/\!\sim~ \models [u] \notin  [v]$,
    \item[(ii)] if $D\setminus \{\one\} \neq \varnothing$ then there exist $[u], [v] \in \VA{\A}/\!\sim$ such that $\VA{\A}/\!\sim~ \models [u] \in  [v]$ and $\VA{\A}/\!\sim~ \models [u] \notin  [v]$ both hold.
\end{enumerate}
\end{teo}

\begin{proof}
(i) Consider $[u], [v] \in \VA{\A}/\!\sim$ so that $\VA{\A}/\!\sim~ \not\models [u] \in  [v]$. By Fact \ref{fact: satisfactions of atomic formulas}(ii), $([u], [v]) \notin R_\in$, which implies that $\VA{\A} \not\models u \in v$, by Definition \ref{Definition: Quotient space identity and belongniness}(ii). 
Hence, $\lb u \in v \rb \notin D$. Therefore by the definition of designated $\co$-algebra, $\lb\neg(u \in v)\rb = \lb u \in v \rb^* = \one \in D$.
This leads to the fact that $\VA{\A} \models \neg (u \in v)$, i.e., $([u], [v]) \in R_{\notin}$, i.e., $\VA{\A}/\!\sim~ \models [u] \notin  [v]$.

(ii) Let us fix an element $a \in D\setminus \{\one\}$. By definition we have $a^* = a$. Now consider any element $u \in \VA{\A}$ and fix the element $v \in \VA{\A}$, where $v = \{\langle u, a \rangle\}$. Then, 
\[\lb u \in v \rb = (v(u) \wedge \lb u = u \rb) = (a \wedge \one) = a \in D.\]
Moreover, $\lb \neg(u \in v) \rb = a^* = a \in D$. Hence, we get that $\VA{\A} \models (u \in v) \wedge \neg (u \in v)$, i.e., $\VA{\A}/\!\sim~ \models [u] \in  [v]$ and $\VA{\A}/\!\sim~ \models [u] \notin  [v]$ both hold.
\end{proof}

The theorem above shows that,  if there exists an element in the designated set other than the top element, then $R_\in \cap R_{\notin} \neq \varnothing$.

\begin{teo}\label{theorem: interpretations of the connectives}
Let $(\mathbb{A},D)$ be an ultra-designated $\co$-algebra. Then, for a sequence $\vec{[u]}=([u_1],[u_2],...)$ in  $\VA{\A}/\!\sim$ we have the following:
    
\begin{enumerate}
\item[(i)] $\VA{\A}/\!\sim~ \models (\varphi \rightarrow \psi) (\vec{[u]})$ iff $\VA{\A}/\!\sim~ \not\models \varphi(\vec{[u]})$ or $\VA{\A}/\!\sim~ \models \psi(\vec{[u]})$,

\item[(ii)] $\VA{\A}/\!\sim~ \models (\varphi \wedge \psi) (\vec{[u]})$ iff $\VA{\A}/\!\sim~ \models \varphi(\vec{[u]})$ and $\VA{\A}/\!\sim~ \models \psi(\vec{[u]})$,

\item[(iii)] $\VA{\A}/\!\sim~ \models (\varphi \vee \psi) (\vec{[u]})$ iff $\VA{\A}/\!\sim~ \models \varphi(\vec{[u]})$ or $\VA{\A}/\!\sim~ \models \psi(\vec{[u]})$,

\item[(iv)] if   $\VA{\A}/\!\sim~ \not\models \varphi(\vec{[u]})$, then $\VA{\A}/\!\sim~ \models \neg \varphi(\vec{[u]}) $,

\item[(v)] $\VA{\A}/\!\sim~ \models  (\forall x_k \varphi)(\vec{[u]})$ iff  $\VA{\A}/\!\sim~ \models \varphi\big{(}\vec{[u]}( d/k) \big{)}$, for all $d \in \VA{\A}/\!\sim$, where $\vec{[u]}(d/k)$ is identified with the sequence $([u_1],\ldots, [u_{k-1}],d,[u_{k+1}], \ldots)$, and

\item[(vi)] $\VA{\A}/\!\sim~ \models  (\exists x_k \varphi)(\vec{[u]})$ iff  $\VA{\A}/\!\sim~ \models \varphi\big{(}\vec{[u]}( d/k) \big{)}$, for some $d \in \VA{\A}/\!\sim$.   
\end{enumerate}
\end{teo}

\begin{proof}

Since $D$ is an ultrafilter of $\A$, by Lemma \ref{lemma: ultrafilter of an MTV algebra} we have that $\mathbf{A} \setminus D = \{\zero\}$, where $\mathbf{A}$ is the underlying set of $\A$.
\begin{align*}
   (i)~~ \VA{\A}/\!\sim~ \models (\varphi \rightarrow \psi) (\vec{[u]}) & \mbox{ iff}~ \VA{\A}~\models (\varphi \rightarrow \psi) (\vec{u})\\
    & \mbox{ iff}~\lb (\varphi \rightarrow \psi) (\vec{u}) \rb \in D\\
    & \mbox{ iff}~\lb (\varphi \rightarrow \psi) (\vec{u}) \rb \neq \zero, \mbox{ since}~\mathbf{A} \setminus D = \{\zero\}\\
    & \mbox{ iff}~ \big{(}\lb \varphi(\vec{u}) \rb \Rightarrow \lb \psi(\vec{u}) \rb \big{)} \neq \zero\\
    & \mbox{ iff}~\lb \varphi(\vec{u}) \rb = \zero \mbox{ or}~\lb \psi(\vec{u}) \rb \neq \zero\\
    & \mbox{ iff}~\VA{\A} \not\models \varphi(\vec{u}) \mbox{ or}~\VA{\A} \models \psi(\vec{u}), \mbox{ applying}~\mathbf{A} \setminus D = \{\zero\}\\
    & \mbox{ iff}~\VA{\A}/\!\sim~ \not\models \varphi(\vec{[u]}) \mbox{ or}~\VA{\A}/\!\sim~ \models \psi(\vec{[u]}).
\end{align*}
The proofs of (ii), (iii), and (iv) can also be done similar to (i), using Lemma \ref{lemma: ultrafilter of an MTV algebra}.
\begin{align*}
    (v)~~ \VA{\A}/\!\sim~ \models  (\forall x_k \varphi)(\vec{[u]}) & \mbox{ iff}~ \VA{\A}~\models (\forall x_k \varphi)(\vec{u})\\
    & \mbox{ iff}~ \bigwedge_{v \in \VA{\A}}\lb \varphi\big{(}\vec{u}(v/k) \big{)} \rb \in D,\\
    & \qquad \mbox{ where $\vec{u}(v/k)$ is the sequence $(u_1,\ldots, u_{k-1},v,u_{k+1}, \ldots)$}\\
    & \mbox{ iff}~\lb \varphi\big{(}\vec{u}(v/k) \big{)} \rb \in D, \mbox{ for all}~v \in \VA{\A},\\
    & \qquad \mbox{ the forward direction is immediate, for the other direction}\\
    & \qquad \mbox{ we use property (iii) of Definition \ref{Definition: MTV algebra} and}~ \mathbf{A} \setminus D = \{\zero\}\\
    & \mbox{ iff}~\VA{\A} \models \varphi\big{(}\vec{u}(v/k) \big{)}, \mbox{ for all}~v \in \VA{\A}\\
    & \mbox{ iff}~\VA{\A}/\!\sim~ \models \varphi\big{(}\vec{[u]}(d/k) \big{)}, \mbox{ for all}~d \in \VA{\A}/\!\sim.
\end{align*}

(vi) It can be proved as (v), because for any subset $\{a_i~:~i \in I\}$ of $\mathbf{A}$, $\bigvee\limits_{i \in I}a_i \in D$ implies that there exists $j \in I$ such that $a_j \in D$ as $\mathbf{A} \setminus D = \{\zero\}$.
\end{proof}

Observe that only the property (iv) of Theorem \ref{theorem: interpretations of the connectives} is not expressed as a necessary and sufficient condition. Indeed, the missing direction of  property (iv) fails. For a  concrete example, consider the formula $\varphi(x, y) := x \in y$ and use Theorem \ref{theorem: belongingness and its negation in the quotient space}(ii). Moreover, we conclude this section, by pointing out that the same exact sentences $\varphi$ and $\psi$ of  Lemma  \ref{lemma: (T, BA) is paraconsistent} witness that the underlying logic of the set theory having a model $\VA{\A}/\!\sim$ is paraconsistent.

\section{Comparison with other models of paraconsistent set theories}

In this section, we compare the paraconsistent models we developed in this paper with models of paraconsistent set theories that we find in the literature. In particular, we will address the models produced by Priest's \textit{model-theoretic} approach given that they are the only paraconsistent models of set theory in which all $\ZF$-axioms are valid. We will argue that our approach overcomes the limits of Priest's model-theoretic approach to paraconsistent set theory. In short we will argue that in the models presented in this paper $\ZF$ holds consistently and not only paraconsistently. By this we mean that in our models $\ZF$ are only valid and not, also, not valid. In this sense, those presented here are the first proper (from a classical perspective) paraconsistent models of $\ZF$.

The model-theoretic approach has been put forward in \cite[Section ~18.4]{Priest2006-PRIICA} and \cite[Section~ 11]{Priest2017-PRIWIT-18}. In order to keep things short, we will assume some familiarity with these papers. The basic idea of the model-theoretic approach of Priest consists in finding an equivalence relation which allows us to collapse a model of $\ZF$ into a model of $\mathsf{NLP}_=\footnote{By $\mathsf{NLP}_{=}$ we denote the set theory that we obtain  by combining
the axioms of na\"ive set theory, i.e., \textsf{Extensionality} and \textsf{Comprehension}
with the logical axioms of LP$_{=}$, where LP$_{=}$ denotes the first-order version of the Logic of Paradox, LP extended with a binary predicate $=$, where $x=y$ receives value $\one$ or $\frac{1}{2}$ just in case $x$ equals $y$ (in the meta-theory).} +\mathsf{ZF}$ and which, moreover, contains a large fragment of the cumulative hierarchy as an inner model. Priest proposes two equivalence relations: the type-lift and the Hamkins type-lift.

Let us sketch one of these constructions. In the case of the type-lift, we start with a model of \textsf{ZF} with two inaccessible cardinals, say, $\kappa_1$ and $\kappa_2$. While everything is preserved below $V_{\kappa_1}$, on the other hand, everything is collapsed between $V_{\kappa_{1}}$ and $V_{\kappa_{2}}$.  The resulting object of such a collapse, call it $a$, is what witnesses the paraconsistency of the model. Instead, the standard hierarchy below $\kappa_1$ is responsible for the validity of \textsf{ZF} in a cumulative hierarchy. 
Applying this construction we obtain the following important result.

\begin{teo}[\cite{Priest2006-PRIICA}]\label{Priests models} Suppose that $\mathcal{M}$ is a classical model of $\mathsf{ZF}$ containing two inaccessible cardinals $\kappa_1$ and $\kappa_2$. Then there is an LP-model $\mathcal{M}^{\sim} = \langle D^{\sim}, I^{\sim}\rangle$  such that:

\begin{enumerate}
    \item[(i)]  $\mathcal{M}^{\sim}$ is a model of \emph{ZF} $+$ $\mathsf{NLP}_{=}$ and
    \item[(ii)] $\mathcal{M}^{\sim}$ contains a model $\mathcal{N}$ where $\mathcal{N}$ is a classical model of $\mathsf{ZF}$.
\end{enumerate}

\end{teo}

For completeness, let us very briefly address the Hamkins type-lift. An important fact about this second construction is that  Theorem \ref{Priests models} also holds for it.  The Hamkins type-lift models constitutes an improvement over the type-lift models given that different sets witness different instances of $\mathsf{Comprehension}$ thus providing more discriminating models; which was not the case for the models obtained by means of the type-lift. Although we have an advance on the set theoretical side, however, this new construction manifests a decisive drawback on the logical side. Indeed, the treatment of  identity is highly non-standard, since Leibniz's law of indiscernibility of identicals fails in the models obtained by the Hamkins type-lift construction. 
On the other hand, for any  $\VA{\A}/\!\sim$, Leibniz's law of indiscernibility of identicals holds. This provides a classical treatment for the notion of identity, as witnessed by Theorem \ref{theorem: validity of LL}. For this reason, we concentrate on the comparison between our models and the original Priest's type-lift models. 

Before presenting the point of difference between the models presented in this paper and the type-lift ones, it is worth reporting Priest's view, who was not completely satisfied with the type-lift constructions. 


\begin{quote}Clearly, the collapsed model just constructed is not a very interesting one.The much over worked $a$ is the witness set for every condition. But it suffices to establish that there are interpretations of both ZF and naive set theory. (\cite{Priest2017-PRIWIT-18}, p. 48)
\end{quote}

We believe that the models discussed here fares better than Priest's in providing a paraconsistent version of $\ZF$ for the following reasons.

\medskip

\begin{enumerate}
    \item  The paraconsistency of the present models does not come from a single set on top of the hierarchy, but it is widespread in the cumulative hierarchy and witnessed by many different sets along the hierarchy (see Theorem \ref{Theorem: models are paraconsistent}). In a slogan, we can say that these algebra-valued models provide a paraconsistent cumulative hierarchy, while Priest's ones a cumulative hierarchy together with a paraconsistent set (which is somewhat isolated from the rest). 
    
    \medskip
    
    \item  The models presented in this paper do not validate the axiom of (unrestricted) \textsf{Comprehension}. Thus, the corresponding set theory is different from that of Priest's models. This is an important point from a classical perspective, since these results show how close we can get to a classical cumulative hierarchy, but with a non-classical logic. 
 
     \medskip

     \item Not only the axioms of \textsf{ZF} are all valid in any  $\VA{\A}/\!\sim$, but it is also important to observe that their negation do not hold in these models. In this case, we can say that \textsf{ZF} holds \emph{consistently}. This is a trivial observation in the classical context, which, however, is not obvious in a paraconsistent context. For example, in the models obtained by a type-lifting construction  the \textsf{Separation Schema} holds together with its negation. In this case we can say that $\ZF$ holds \emph{paraconsistently}.

\end{enumerate}

\medskip

\noindent This last point has been objects of criticism. Very recently Incurvati wrote the following. 

\begin{quote}Indeed, as Priest (2017: 98) points out, once inconsistent sets enter the picture, it can be shown that if ZF’s Separation Schema is true, it is both true and false. What this brings to light is that the paraconsistent set theorist needs, after all, to say more about what the universe of sets looks like. It is not enough to simply suppose that it contains some paradigmatic inconsistent sets and has the cumulative hierarchy as an inner model. (\cite{Incurvati2020-INCCOS}, p.126) \end{quote} 

We believe that the structures built in this paper address exactly the issue raised by the above quotation. Not only we have paraconsistent models where $\ZF$ holds consistently, but the similarity these structures bear with $\textbf{V}$ leave no mystery with respect to what they look like: a very familiar cumulative hierarchy. 




Since these new structures satisfy both requirements of being  paraconsistent and being non-trivial models of $\ZF$, i.e., validating these axioms consistently (as witnessed by the validity of Theorem 3.11 and Theorem 4.5), we can claim that they are the first example of non-trivial paraconsistent models of $\ZF$.


\section{Generalising the approach}

We end the paper with a few questions that the present approach to non-classical models of set theory suggests. Besides the construction of non-trivial paraconsistent models of set theory in which $\ZF$ hold consistently, this paper introduces a novelty in the algebra-valued constructions: the change of the interpretation function. We applied this idea here to the construction of paraconsistent models of $\ZF$, but nothing prevents us to generalise the technique to fit other logical environments.\footnote{Some interesting preliminary results have been obtain with respect to Heyting-valued models. We leave to a next time a full discussion of this other case.} 
Table \ref{table:in B-valued models} and Table \ref{table:in W-algebra-valued models} give an overview of the study made in this paper, which help to analyze the open questions raised in this section.

\begin{table}[H]
    \centering
    \begin{tabular}{|c|c|c|c|c|}
  \hline
  & Axioms of $\ZF$ & Axioms of $\overline{\ZF}$ & $\llaw$ & $\mathbf{L}(\lb\cdot\rb, \VA{\A})$ \\
  \hline
  $\lb \cdot \rb_\ba$ & Valid & Valid & Valid & Not known in general.\\
  & & & & Classical, for atomic Boolean algebras. \\
  \hline
  $\lb \cdot \rb_\pa$ & Valid & Valid & Valid & Not known in general.\\
  & & & & Classical, for atomic Boolean algebras. \\
  \hline
    \end{tabular}
    \caption{With respect to Boolean-valued models, $\VA{\A}$.}
    \label{table:in B-valued models}
\end{table}

\begin{table}[H]
  \centering
     \begin{tabular}{|c|c|c|c|c|}
  \hline
  & Axioms of $\ZF$ & Axioms of $\overline{\ZF}$ & $\llaw$ & $\mathbf{L}(\lb\cdot\rb, \VA{\A})$ \\
  \hline
  $\lb \cdot \rb_\ba$ & $\nff$-$\ZF$ is valid & $\nff$-$\ZF + \overline{\mathsf{Extensionality}}$ & Fails for some $\varphi$ & Paraconsistent \\
  &  & is valid & &\\
  \hline
  $\lb \cdot \rb_\pa$ & $\ZF - \mathsf{Extensionality}$ & Valid & Valid & Paraconsistent \\
  & is valid & & &\\
  \hline
    \end{tabular}
    \caption{With respect to ultra-designated $\co$-algebra valued models, $\VA{\A}$, where $\A$ has more than two elements.}
    \label{table:in W-algebra-valued models}
\end{table}

The first one deals with the scope of the validity of the Zermelo-Fraenkel axioms and their correct formulation. Indeed,  notice that $\lb \cdot \rb_\ba$ only provides Boolean-valued models of $\ZF$, while $\lb \cdot \rb_\pa$ both Boolean and ultra-designated $\co$-algebra-valued models of $\overline{\ZF}$. From this perspective both $\lb \cdot \rb_\pa$ and $\overline{\ZF}$ are more inclusive since they capture a larger class of models for Zermelo-Fraenkel set theory. 

\medskip
\textbf{Question 1} Is there a maximal class $\mathcal{X}$ of algebras and a couple $(\ZF^{X}, \lb \cdot \rb_{\mathrm{X}})$, with $\ZF^{X}$ an axiomatic system classically equivalent to $\ZF$ and $\lb \cdot \rb_{\mathrm{X}}$ an interpretation function, such that  $\VA{\mathbb{X},~\lb \cdot \rb_{\mathrm{X}}} \models \ZF^{X}$, for $\mathbb{X} \in \mathcal{X}$?
\medskip

Notice that there is no obvious reason why this attempt of maximising the class of models that validate an appropriate version of $\ZF$ should yield a unique result. If this is indeed the case, however, we would obtain an interesting relativization of the notion of independence. 

\medskip
\textbf{Question 2} Assume that there exist two couples $(\ZF^{X}, \lb \cdot \rb_{\mathrm{X}})$ and $(\ZF^{Y}, \lb \cdot \rb_{\mathrm{Y}})$ that define two maximal but incompatible classes of models $\mathcal{M}_{\mathcal{X}}$ and $\mathcal{M}_{\mathcal{Y}}$, each validating the appropriate version of $\ZF$ and both extending the class of Boolean-valued models. Then, is it possible to find a sentence in the pure language of set theory, $\mathcal{L}_{\in}$, that is independent from $\ZF^{X}$ but not from $\ZF^{Y}$? In other terms, is it possible to show that the notion of independence is not absolute? 
\medskip

An interesting aspect of the above question is that it does not necessarily need a non-classical logical environment. As a matter of fact, once we start modifying the interpretation function, we can raise the same issue in a purely classical context.

\medskip
\textbf{Question 2$^*$}
Is there an assignment function $\lb \cdot \rb_{\mathrm{BA}'}$ from a $\mathcal{L}_{\mathbb{B}}$-language into a Boolean algebra $\mathbb{B}$, such that all $\ZF$ axioms receive value $\one$, but where there are formulas in the pure language of set theory, $\mathcal{L}_\in$,  which receive value $\one$ under $\lb \cdot \rb_{\mathrm{BA}'}$, but not under $\lb \cdot \rb_{\ba}$?
\medskip

Another interesting question about these non-classical models is whether they can be used to produce (possibly new) classical models. 

\medskip
\textbf{Question 3}\footnote{We thank Joel D. Hamkins for this interesting suggestion.}
Is it possible to define a two step iteration in which the first step is non-classical, but the second is classical? More concretely, is it possible to find a Boolean algebra $\mathbb{B}$ within a ultra-designated $\co$-algebra-valued model $\VA{\mathbb{T},~\lb \cdot \rb_{\pa}}$ and to build from there a Boolean-valued model $\big{(}\VA{\mathbb{T},~\lb \cdot \rb_{\pa}}\big{)}^{(\mathbb{B},~\lb \cdot \rb_{\ba})}$? If this is the case, is it possible to obtain this new Boolean-valued model as a Boolean-valued extension of $\mathbf{V}$? In other terms, is $\big{(}\VA{\mathbb{T},~\lb \cdot \rb_{\pa}}\big{)}^{(\mathbb{B},~\lb \cdot \rb_{\ba})}$ a new classical model? 
\medskip

Finally, it would be interesting to precisely determine  the extension of the class of models that we investigated in this paper.

\medskip
\textbf{Question 4}
Is there an algebraic way to  describe the class of structures that validate  $\overline{\ZF}$ under $\lb \cdot \rb_\pa$?
\medskip



\medskip

\noindent \textbf{Acknowledgments}.
 The first
author acknowledges support from the FAPESP grant. n. 2017/23853-0. The second author wants to acknowledge FAPESP for providing him a Visiting Researcher grant (n. 2016/25891-3) to spend one year at
the Philosophy Department of the University of Campinas. The third author
acknowledges support from FAPESP, Jovem Pesquisador grant (n. 2016/25891-
3), and from CNPq grant (n. 301108/2019-6).

\bibliographystyle{apalike}
\bibliography{bibliografia.bib}

\end{document}